\documentclass[11pt,reqno, footinclude=true, headinclude=true, openright]{article}
\pdfoutput=1
\usepackage[utf8]{inputenc}
\usepackage[margin=1in]{geometry}
\usepackage{enumerate}
\usepackage{color, xcolor, colortbl}
\usepackage{graphicx,epstopdf}
\usepackage{amsmath,amssymb,amsthm}
\usepackage{algorithm}
\usepackage{algorithmic}
\usepackage{esvect}
\usepackage{bm}
\usepackage[caption=false]{subfig}
\usepackage{appendix}
\usepackage{multirow}
\usepackage{verbatim}
\usepackage{mathtools}
\usepackage{dsfont}
\usepackage{braket}
\usepackage[english]{babel}
\usepackage{csquotes}
\usepackage{caption}

\usepackage{tikz}
\usetikzlibrary{arrows,calc} \usetikzlibrary{decorations,arrows}
\usetikzlibrary{decorations.pathreplacing}
\usetikzlibrary{arrows.meta}
\usetikzlibrary{positioning}
\usetikzlibrary{quantikz2}

\usepackage[colorlinks=true,linkcolor=red,bookmarks=true,breaklinks=true]{hyperref}
\usepackage[capitalize,noabbrev]{cleveref}

\DeclareFontEncoding{FMS}{}{}
\DeclareFontSubstitution{FMS}{futm}{m}{n}
\DeclareFontEncoding{FMX}{}{}
\DeclareFontSubstitution{FMX}{futm}{m}{n}
\DeclareSymbolFont{fouriersymbols}{FMS}{futm}{m}{n}
\DeclareSymbolFont{fourierlargesymbols}{FMX}{futm}{m}{n}
\DeclareMathDelimiter{\VERT}{\mathord}{fouriersymbols}{152}{fourierlargesymbols}{147}

\newcommand{\C}{{\mathbb{C}}}
\newcommand{\R}{{\mathbb{R}}}

\newtheorem{theorem}{Theorem}[section]

\newtheorem{definition}[theorem]{Definition}

\newtheorem{lemma}[theorem]{Lemma}

\newtheorem{proposition}[theorem]{Proposition}
\newtheorem{remark}[theorem]{Remark}

\PassOptionsToPackage{USenglish}{babel}
\usepackage[USenglish]{babel}

\numberwithin{equation}{section}

\DeclareMathOperator{\supp}{supp}

\newcommand{\cH}{{\mathcal H}}
\newcommand{\cE}{{\mathcal E}}
\newcommand{\cO}{{\mathcal O}}
\newcommand{\cU}{{\mathcal U}}
\newcommand{\cF}{{\mathcal F}}

\newcommand{\cP}{{\mathcal P}}

\newcommand{\IR}{{\mathbb R}}
\newcommand{\IC}{{\mathbb C}}
\newcommand{\IN}{{\mathbb N}}
\newcommand{\IU}{{\mathcal U}}
\newcommand{\IQ}{{\mathbb Q}}
\newcommand{\IP}{{\mathbb P}}

\newcommand{\dom}{\mathrm{D}}

\newcommand{\bsb}{{\boldsymbol b}}
\newcommand{\bsf}{{\boldsymbol f}}
\newcommand{\bsy}{{\boldsymbol y}}
\newcommand{\bsc}{{\boldsymbol c}}
\newcommand{\bsrho}{{\bm \rho}}
\newcommand{\bsnu}{{\bm \nu}}
\newcommand{\bsmu}{{\bm \mu}}
\newcommand{\bse}{{\bm e}}
\newcommand{\bsalpha}{{\bm \alpha}}

\usepackage[style=numeric,url=false,sorting=none,backref=true,maxbibnames=99,doi=false]{biblatex}
\begin{document}

\title{
Quantum Circuit Encodings of Polynomial Chaos Expansions}

\author{Junaid Aftab$^{1}$, \; Christoph Schwab$^{2}$, \; Haizhao Yang$^{1,3}$, \; Jakob Zech$^4$\\
	\footnotesize $^{1}$ Department of Mathematics,  University of Maryland, College Park, MD 20742, USA\\
	\footnotesize $^{2}$ Seminar for Applied Mathematics, ETH Z{\"u}rich, R{\"a}mistrasse 101, CH-8092 Z{\"u}rich, Switzerland\\
	\footnotesize $^{3}$ Department of Computer Science, University of Maryland, College Park, MD 20742, USA \\
	\footnotesize $^{4}$ Interdisciplinary Center for Scientific Computing, Heidelberg University, Heidelberg, 69120, Germany}

\date{}

\maketitle

\begin{abstract}
This work investigates the expressive power of quantum circuits in approximating high-dimensional, real-valued functions. 
We focus on countably-parametric holomorphic maps $u:U\to \mathbb{R}$, where the parameter domain is $U=[-1,1]^{\mathbb{N}}$. 
We establish dimension-independent quantum circuit approximation rates 
via the best $n$-term truncations of generalized polynomial chaos (gPC) expansions of these parametric maps, 
demonstrating that these rates depend solely on the summability exponent of the gPC expansion coefficients. 
The key to our findings is based on the fact that so-called ``$(\bsb,\epsilon)$-holomorphic'' functions, 
where $\bsb\in (0,1]^\mathbb N \cap \ell^p(\mathbb N)$ for some $p\in(0,1)$, permit structured and sparse gPC expansions. 
Then, 
$n$-term truncated gPC expansions are known to admit
approximation rates of order $ n^{-1/p + 1/2}$ in the $L^2$ norm 
and 
of order $ n^{-1/p + 1}$ in the $L^\infty$ norm.
We show the existence of parameterized quantum circuit (PQC) encodings of these $n$-term truncated gPC expansions, and bound PQC depth and width via (i) tensorization of univariate PQCs that encode 
Chebyshev-polynomials in $[-1,1]$ and (ii) linear combination of unitaries (LCU) to build PQC emulations of $n$-term truncated gPC 
expansions. The results provide a rigorous mathematical foundation for the use of quantum algorithms in high-dimensional function approximation. 
As countably-parametric holomorphic maps naturally arise in parametric PDE models and uncertainty quantification (UQ), 
our results have implications for quantum-enhanced algorithms for a wide range of maps in applications.
\end{abstract}

\vspace{0.2cm}

\vspace{0.2cm}

\tableofcontents

\section{Introduction}
\label{sec:Intro}
%
Quantum computing has emerged as a transformative paradigm for addressing computationally demanding problems, 
including high-dimensional partial differential equations (PDEs) 
\cite{childs2021high, liu2021efficient,JIN2022111641}.  
Nevertheless, harnessing these advantages on current noisy intermediate-scale quantum (NISQ) devices necessitates the development of alternative approaches that are compatible with limited qubit counts and do not rely on full quantum error correction. 

PQCs have emerged as a promising framework for function approximation, both 
in scientific computing and machine learning~\cite{benedetti2019pqc, XiaodongDing, Sim_2021}. 
Viewed as quantum analogs of classical neural networks, 
PQCs offer a flexible and expressive approach to approximating high-dimensional functions and, in some settings, 
may outperform classical methods in terms of computational efficiency~\cite{schuld2019qml, havlicek2019supervised}. 
A central challenge, however, lies in rigorously quantifying the approximation capabilities of PQCs 
and establishing the precise circumstances under which they surpass classical 
methodologies such as polynomial approximations, wavelets, 
and deep neural networks~\cite{yang2021nnapproxsmooth, devore2021neural,SHEN2022101}.

Classical approximation theory provides a well-developed foundation for analyzing the expressive power of 
approximation architectures, 
including polynomials, neural networks, and kernel methods~\cite{yang2021nnapproxsmooth, devore2021neural}. 
To enable a meaningful comparison between classical and quantum approaches, 
it is essential to develop a similar theoretical framework for quantum circuits. 
In this work, we analyze the expressive power of quantum circuits for approximating 
high-dimensional, real-valued functions, which correspond to 
countably-parametric holomorphic maps \( u: U \to \mathbb{R} \) with \( U = [-1,1]^{\mathbb{N}} \). 
Focusing on \((\bsb, \epsilon)\)-holomorphic functions, 
where \( \bsb \in (0,1]^{\mathbb{N}} \cap \ell^p(\mathbb{N}) \) for \( p \in (0,1) \), 
we leverage their sparse generalized polynomial chaos (gPC) expansions to construct 
parameterized quantum circuits via tensorization of (block-encodings of) Chebyshev polynomials
within the Linear Combination of Unitaries (LCU) framework of \cite{Childs2012lcu,wang2021fast}.
The LCU PQC architecture has been identified recently as 
promising for near-term implementation \cite{chakraborty2023LCU-NearTerm}. 

\subsection{Contributions}
\label{sec:Contr}
The encoding and approximation of polynomial chaos expansions have recently attracted significant attention, particularly in the context of approximating parametric solution families of partial differential equations (PDEs) with uncertain inputs, as discussed in \cref{sec:ExplRspSrf}. 
A key insight from this literature is that the sparsity (as expressed via the summability index $p\in (0,2)$) 
of the coefficient sequence in a polynomial chaos expansion 
determines the mean-square rate of finite-parametric approximations via \(n\)-term truncations.

Our main contributions are as follows. We review that fact that 
certain very high-dimensional functions—specifically, 
infinite-parametric, holomorphic functions depending on 
countably many parameters — can be approximated well by truncating their polynomial chaos expansions. 
Such functions have been shown recently to arise in the context of response surfaces of elliptic and parabolic PDEs,
which are subject to uncertain input data.
We rigorously establish, for PQCs in the LCU framework,
dimension-independent approximation rates for countably-parametric holomorphic functions 
\( u : U \to \mathbb{R} \) defined on infinite-dimensional parameter domains \( U = [-1,1]^{\mathbb{N}} \). 
Under the assumption that these functions are \((\boldsymbol{b}, \epsilon)\)-holomorphic 
with coefficient sequences 
\(\boldsymbol{b} \in (0,1]^{\mathbb{N}} \cap \ell^p(\mathbb{N})\) for some \( p \in (0,1) \), 
we prove that their best \( n \)-term truncated gPC expansions achieve approximation rates of order \( n^{-1/p + 1/2} \) 
in the \( L^2 \)-norm and order \( n^{-1/p + 1} \) in the \( L^\infty \)-norm, 
which orders are independent of the parameter dimension. 
Importantly, the quality of this approximation does not deteriorate 
as the number of parameters grows. 
The rate at which the approximation improves depends on how quickly the coefficients in these expansions decrease, 
captured by a summability condition on the coefficients. Simply put, if the coefficients are sufficiently sparse, 
then we can approximate the function efficiently 
with only a limited number of terms, and this approximation works well both on average and in the worst case.

Our results and proofs are constructive: 
we build explicit PQC encodings of the truncated Chebyshev- and Taylor-GPC expansions. 
This is achieved by tensorizing univariate PQCs that block-encode Chebyshev polynomials on \([-1,1]\) 
and applying the linear combination of unitaries (LCU) technique. 
In particular, the resulting multivariate PQCs provide, as tensorized LCUs, stability. 
We provide rigorous bounds on the PQC circuit depth and width required to implement these approximations. 
Our findings establish a solid mathematical foundation for leveraging quantum algorithms in approximating high-dimensional parametric maps, with direct implications for parametric PDEs and uncertainty quantification. This work thereby paves the way for quantum-enhanced algorithms targeting complex infinite-dimensional problems in scientific computing.

In summary, this work contributes to closing this gap by analyzing the expressive power of quantum circuits for high-dimensional function approximation, characterizing the trade-offs among circuit depth, entanglement, and approximation accuracy, and laying the groundwork for a rigorous comparison between classical and quantum models. 
We emphasize that the quantum circuit constructions employed here are largely adapted from existing designs; the primary contribution of this work is the accompanying infinite-dimensional approximation theory. 
This demonstrates that, in the finite-dimensional setting, PQC representations inherit the classical exponential convergence behavior of polynomial approximations, leading to highly efficient quantum circuit constructions. The novelty of this work lies not in the construction of quantum circuits per se, but in establishing the first dimension-independent approximation theory for PQCs applied to infinite-dimensional parametric maps.

\subsection{Example: Response Surface of Parametric, Elliptic PDEs}
\label{sec:ExplRspSrf}
Let us present a typical example of the infinite-parametric maps.
In an open and bounded Lipschitz domain $\dom \subseteq \IR^n$, 
and given sufficiently regular functions $f$ and $a:\dom \to \IR$, 
consider the diffusion equation
\begin{equation}\label{eq:Diff}
	f + \nabla_x\cdot (a(x)\nabla_x u) = 0 \quad \mbox{in} \quad \dom\;,\quad u|_{\partial\dom} = 0 \;.
\end{equation}
It is well-known that for coefficient functions 
that are continuous and bounded in $\overline{\dom}$ and satisfy $a(x) \geq a_- > 0$ for $x\in \dom$, 
there exists a unique, weak solution $u\in H^1_0(\dom)$ of \cref{eq:Diff}. 
Assuming that $f$ and $\dom$ are given, fixed,
we shall be interested in expressing the \emph{solution map} $S:a\mapsto u$,
on a suitable set of admissible coefficients $a$.
To describe such a set, 
introduce the set of \emph{affine-parametric diffusion coefficients} 
in \cref{eq:Diff} via the \emph{affine-parametric expansion}
\begin{equation}\label{eq:DiffPar}
	a(x;\bsy) := a_0(x) + \sum_{j\geq 1} y_j \psi_j(x) \;\;, 
\end{equation}
where $|y_j|\leq 1$, 
$a_0, \psi_j \in C(\overline{\dom})$ with $a_0(x) \geq a_{0,-} > 0$ for $x\in D$, 
and where we assume that there exists $\kappa \in (0,1)$ such that
\begin{equation}\label{eq:a>0}
	\| \bsb \|_{\ell^1(\IN)} 
	\leq \kappa a_{0,-} \;, \text{ where }\bsb = (b_j)_{j\geq 1} \; \text{with }
	b_j := \| \psi_j \|_{L^\infty(\dom)}. 
\end{equation}
Typical examples of expansions \cref{eq:DiffPar} are Karhunen-Loeve 
expansions of random field inputs, Fourier- or wavelet-expansions. 
Under Assumption \cref{eq:a>0}, 
for every $\bsy = (y_j)_{j\geq 1} \in \cU := [-1,1]^\IN$ holds
\begin{equation} \label{eq:a(x,y)>0}
	\forall x\in \overline{\dom} \text{ and } \forall \bsy\in \cU: 
	\quad a(x,\bsy) \geq a_{0,-} - \sum_{j\geq 1} b_j \geq (1-\kappa)a_{0,-} > 0
	\;.
\end{equation}
By the Lax-Milgram lemma, 
for given, fixed $f\in H^{-1}(\dom) = (H^1_0(\dom))^*$ 
and for every $\bsy\in \cU$, 
the boundary value problem \cref{eq:Diff} admits a unique solution
$u(\cdot,\bsy)\in H^1_0(\dom)$, 
the Hilbertian Sobolev space of functions with square integrable
weak first order derivative which vanish on $\partial\dom$.
The parametric family of solutions $u:\cU\to H^1_0(\dom)$ is holomorphic with respect to the parameters $\bsy$ \cite{CCS15}.
Let $G\in H^{-1}(\dom)$. 
Then the map $F = G \circ u$ i.e.
\begin{equation}\label{eq:PDErspsrf}
	F:\cU\to \IR: \bsy \mapsto G(u(\cdot,\bsy))
\end{equation}
is a countably-parametric, holomorphic map, which is referred also as
``response surface'' of the PDE \cref{eq:Diff} for the 
parametric input \cref{eq:DiffPar} and the ``quantity of interest'' $G(\cdot)$. 
Response maps of the type \cref{eq:PDErspsrf} with quantified holomorphic parameter dependence 
as we shall assume below occur for a wide range of elliptic and parabolic PDE models with uncertain
input in engineering and in the sciences.
Given a target accuracy $\epsilon>0$,
we are interested in quantum circuits $F_\epsilon: \cU\to \IR$ 
expressing $F$ to accuracy $\epsilon>0$,
i.e. 
finding a PQC $F_\epsilon$ that provides an 
$\epsilon$-accurate emulation of the response 
map $F:\cU\to \IR$, i.e., such that 
$$
\sup_{\bsy\in \cU} | F(\bsy) - F_\epsilon(\bsy) | \leq \epsilon.
$$
In more particular, we will establish a quantitative estimate of the PQC width and depth to achieve the target accuracy, where width refers to the number of qubits the circuit operates on, while depth represents the number of time steps or layers of gates needed to execute the circuit. This is an important question of numerical analysis in practical applications of quantum computing.
%

\subsection{Related Works}
We briefly discuss some related works in both neural network and quantum computing literature.
{}
\subsubsection{Neural Network Literature} 
\label{sec:NNLit}

A parameterized quantum circuit can be viewed as the quantum counterpart of a classical feedforward neural network. In the classical setting, substantial progress has been made toward a rigorous theoretical understanding of deep neural networks (DNNs). DNN approximation theory \cite{devore2021neural} provides a systematic framework for analyzing their expressive power, focusing on establishing approximation error bounds with respect to network width, depth, and the total number of neurons \cite{yang2021nnapproxsmooth, yang2022optimalapproxrelu}. This line of research aims to identify the intrinsic limitations of neural networks in approximating functions from various function spaces, independent of specific training algorithms or data availability. It has been shown that DNNs outperform traditional approximation tools in classical function spaces, such as the space of continuous functions \cite{CiCP-28-1768,yang2022optimalapproxrelu,Shen2021ThreeLayers} and Sobolev spaces \cite{yang2023nearly, yang2023nearlysobolev}. In particular, for $d$-dimensional continuous functions, DNNs with advanced activation functions can achieve approximation accuracy using a number of parameters that scale polynomially in $d$ \cite{Shen2021DeepNetwork,Shen2021ThreeLayers,JMLR:v23:21-1404}, thereby avoiding the curse of dimensionality. Moreover, DNNs with simple activation functions can attain dimension-independent approximation rates when the target function belongs to a reduced function space, such as the Barron function class \cite{barron1993universal,e2019apriori}, certain PDE solution spaces \cite{pmlr-v145-e22a,luo2024two,chen2021representation,hutzenthaler2020proof,grohs2023proof,gonon2023deep}, or sparsity-admitting function spaces as studied in the context of polynomial chaos expansions \cite{OSZ2022,ChSJZDlHd19,DDVKNDTP23}. 

\subsubsection{Quantum Computing Literature}
Recent developments in quantum approximation theory have culminated in the formulation of universal approximation theorems for PQCs, establishing their capacity to approximate arbitrary target functions to a prescribed accuracy \cite{perez2021universal, goto2021qml}. These results mirror the classical universal approximation theorems for neural networks, while incorporating distinct quantum-specific elements such as entanglement structure, quantum circuit width, depth, and the impact of measurement processes on function representation \cite{yu2022power}. In addition, recent research has explored the complexity of quantum circuits required for function approximation, emphasizing the role of entanglement and quantum feature maps in enhancing expressivity. Advances in quantum numerical analysis suggest that certain function classes can be efficiently approximated using quantum-enhanced techniques. In particular, methods such as quantum signal processing (QSP) \cite{low2017qsp,gilyen2019qsvt} and linear combinations of unitaries (LCU) \cite{Childs2012lcu, childs2015lcutaylor} have facilitated the construction of quantum circuits that efficiently approximate high-dimensional function spaces~\cite{yu2023provable}. These approaches leverage quantum coherence to achieve improved convergence rates, revealing potential quantum advantages in numerical analysis and scientific computing. Furthermore, quantum circuits designed using Fourier and Chebyshev polynomial expansions~\cite{aftab2024korobov} have shown the ability to approximate both smooth and non-smooth functions with accuracy comparable to-and in some cases surpassing-that of classical methods. These findings indicate that quantum circuits may offer benefits when approximating functions with periodic or structured characteristics, a property utilized in quantum algorithms for signal processing and spectral estimation~\cite{montanelli2019relu}.

\subsection{Organization}
\label{sec:layout}
The remainder of this paper is structured as follows.
\cref{sec:Prel} reviews the notational conventions adopted throughout this work and summarizes essential background on gPC expansions using tensorized Chebyshev polynomials, as well as foundational concepts in quantum computing.
\cref{sec:QAprgPC} presents our main results, detailing approximation error bounds and complexity estimates. 
\cref{sec:Concl} concludes with a summary of our findings and directions for future research.
\section{Preliminaries}
\label{sec:Prel}
To make the presentation self-contained, this section provides the necessary notation and preliminary results. Standard notation and terminology from numerical analysis are introduced in \cref{sec:notation}. 
Relevant results concerning the polynomial approximation of multivariate holomorphic functions and gPC expansions are summarized in \cref{sec:gPC}. Our main analysis relies on explicit encodings using Chebyshev{} polynomial systems and their associated chaos expansions. Hence, the corresponding key definitions and scaling properties are reviewed, with particular emphasis on tensorized constructions suited for high- and infinite-dimensional settings. Elements of quantum signal processing and the LCU method for quantum circuit design are summarized in \cref{sec:Qc}, following \cite{Childs2012lcu,gilyen2019qsvt}.

\subsection{Notation and Terminology}
\label{sec:notation}
We review some relevant notation and terminology of numerical analysis in this section.
%
\subsubsection{Standard Notation}
\label{sec:StndNot}
Let $\mathbb R, \mathbb C$, and $\mathbb N$ denote 
the set of real numbers, complex numbers, and natural numbers, respectively.
We write $\IN_0$ for the set $\IN\cup \{0\}$.
For $z \in \mathbb C$, $z^*$ denotes the complex conjugate of $z$. 
For a subset $A \subseteq \mathbb R, \mathbb C$, $\mathds 1_A(z)$ denotes the indicator function of the set $A$:
\begin{equation*}
	\mathds 1_A(z) = 
	\begin{cases}
		1, & \text{if } z \in  A, \\
		0, & \text{if } z \notin A.
	\end{cases}
\end{equation*}
With $\mathbb R[x]$ and $\mathbb C[x]$, we 
denote the polynomial ring in one variable over $\R$ and $\C$, respectively. 
For $p(x) \in \mathbb C[x]$, 
$p^*(x)$ denotes the polynomial in $\mathbb C[x]$ 
which is obtained by taking the complex conjugate of each coefficient of $p(x)$. For non-negative functions $f,g : \IN \to [0,\infty)$, 
we use the following standard notation from complexity theory. The usual asymptotic order notation is in force:
\begin{enumerate}
	\item $f(n) = \mathcal O (g(n))$ if and only if 
	there exists a constant $C > 0$ and $N \in \mathbb N$ 
	such that $f(n) \leq C g(n)$ for $n \geq N$.
	\item $f = \Omega(g)$ if and only if there exists a $C > 0$ and $N \in \mathbb N$ 
	such that $f(n) \geq C g(n)$ for $n \geq N$.
	\item $f(n) = \Theta(g(n))$ if and only if $f(n) = \mathcal O (g(n))$ and $f(n)=\Omega(g(n))$.
\end{enumerate}
Here, the symbol $C > 0$ will stand for a generic,
positive constant that is independent of any quantities
determining the asymptotic behaviour of an estimate.
It may change its value even within the same equation.

\subsubsection{Multi-Index Notation}
\label{sec:MultIndNot}
A $\bsalpha\in \IN_0^d$ for some finite $d \geq 1$ is a multi-index of finite length.
For multi-indicies of finite length, $\bsalpha$ and $\boldsymbol{\beta}$, component-wise arithmetic operations are used:
\begin{align*}
	\boldsymbol{\alpha} \cdot \boldsymbol{\beta} &= (\alpha_1 \beta_1, \ldots,  \alpha_d\beta_d), \\
	b \cdot \boldsymbol{\alpha} &= (b\alpha_1, \ldots, b\alpha_d), \\
	c^{\boldsymbol{\alpha}} &= (c^{\alpha_1}, \ldots, c^{\alpha_d}), \quad \text{for } b, c \in \mathbb{R}.
\end{align*} 
We will also consider multi-indicies $\bsnu=(\nu_j)_{j\in\mathbb{N}}\in\IN_0^{\IN}$,
of possibly infinite length. 
The notation $\supp(\bsnu) \subseteq \IN$ stands for the 
\emph{support} of the multi-index, i.e. 
\[
\supp(\bsnu) = \set{j\in\IN \mid \nu_j\neq 0}.
\]
The \emph{support size} of $\bsnu$ is 
$| \bsnu |_0 := \#(\supp(\bsnu) ) \leq \infty$,
its \emph{total order} 
is $|\bsnu|_1:=\sum_{j\in \mathbb{N}} \nu_j$,
and 
its \emph{maximal order} is $|\bsnu|_\infty := \sup\{\nu_j | j\in \IN\}$.
In our discussion of polynomial chaos expansions, we use the (countable) set $\cF$ of ``finitely supported'' multi-indices 
\begin{equation*}
	\cF
	:=
	\set{\bsnu\in\IN_0^{\IN} \mid |\bsnu|_1 < \infty } .
\end{equation*}
A subset $\Lambda\subseteq \cF$ is called \emph{downward closed} (``d.c.'' for short),
if
$\bsnu=(\nu_j)_{j\in\mathbb{N}}\in\Lambda$ 
implies
$\bsmu=(\mu_j)_{j\in\mathbb{N}}\in\Lambda$ 
for all $\bsmu\le\bsnu$.
Here, 
the partial-ordering ``$\le$'' on $\cF$ is defined as
$\mu_j\le\nu_j$, for all $j\in \IN$. A subset $\Lambda\subseteq \cF$ is called \emph{anchored}
if for every $j\in \IN$ holds
$\bse_j \in \Lambda$ implies that $\{\bse_1, \bse_2, ..., \bse_j\} \subseteq \Lambda$. 
Here, $\bse_j\in \ell^\infty(\IR^d)$ denotes the Kronecker sequence.
We consider the sets $\IR^\IN$ and $\IC^\IN$ 
endowed with the product topology.
Any subset such as $[-1,1]^{\mathbb{N}}$ is then understood
to be equipped with the subspace topology. For $\epsilon\in (0,\infty)$ we write 
$$B_\epsilon:=\set{z\in\IC | |z|<\epsilon}.$$
Furthermore
$B_\epsilon^{\mathbb{N}} := \bigtimes_{j\in\mathbb{N}} B_\epsilon \subseteq
\mathbb{C}^{\mathbb{N}}$.
Elements of $\IR^\IN$ and of $\IC^\IN$ 
will be denoted by boldface characters such as
$\bsy=(y_j)_{j\in\mathbb{N}} \in [-1,1]^{\mathbb{N}}$. 
For $\bsnu\in \cF$, standard notations
$\bsy^{\bsnu}:=\prod_{{j\in\mathbb{N}}} y_j^{\nu_j}$ and
$\bsnu!=\prod_{{j\in\mathbb{N}}}\nu_j!$ will be employed (observing
that these formally infinite products contain only a finite number of
nontrivial factors with the conventions $0! := 1$ and $0^0 := 1$).
\subsubsection{Function Space Notation}
\label{sec:FctNot}
For $r \in \mathbb N\cup \{0\}$, we denote by $\tilde{T}_r(x)$ the \emph{classical first-kind Chebyshev-polynomials}:
they are defined recursively such that $\tilde{T}_0(x)=1$, $\tilde{T}_1(x)=x$, 
and
\begin{align*}
	\tilde{T}_r(x) = 2x\tilde{T}_{r-1}(x) - \tilde{T}_{r-2}(x), \quad r \geq 2,
\end{align*}
on $ x\in [-1,1]$.
We write $T_k(x)$ for the \emph{non-classical, ``physical'' Chebyshev-polynomials} 
which are normalized in their weighted $L^2$ space (see below). For finite $d\in \IN$, and 
a bounded domain $\dom \subseteq \R^d$,
let $(\dom, \mathcal{B} ,d \boldsymbol{x})$  
be the usual measure space with the $d$-dimensional Lebesgue measure. 
For $1 \leq p \leq  \infty$, 
the Lebesgue function space, $L^p(\dom;\R) $, 
of measurable maps $f:\dom\to \R$ which are $p$-integrable 
is defined as,
\[
L^p(\dom) = \{f : X \rightarrow \mathbb{R} \,|\, 
f \text{ is Lebesgue measurable and } \|f\|_{L^p(\dom)} < \infty\}.
\]
The corresponding (quasi-)norms are
\begin{align*}
	\|f\|_{L^p(\dom)} &:= \left( \int_{\dom} |f(\boldsymbol{x})|^p \, d \boldsymbol{x} \right)^{1/p}, \\
	\|f\|_{L^\infty(\dom)} &:= \operatorname*{ess\,sup}_{\boldsymbol{x} \in \dom} |f(\boldsymbol{x})|.
\end{align*}
We also considered weighted $L^2$ space. In particular, let $\varrho^{(1)}$ denote the probability Chebyshev-measure
on the univariate domain $\IU = [-1,1]$.
It holds, 
with $dy$ denoting the univariate Lebesgue-measure in $\IU$,
that
$$
d\varrho^{(1)}(y) = \frac{1}{\pi\sqrt{1-y^2}} dy,  \quad y\in \IU\;.
$$
For a finite dimension $d\in \IN$, we denote
$$
\varrho = \varrho^{(d)} = (\varrho^{(1)})^{\otimes d} 
$$
the $d$-fold tensor product of the univariate Chebyshev-measure,
on the parameter domain $\IU = [-1,1]^d$. 
When $d = \infty$, the Kolmogorov extension theorem implies
the existence of a unique (countable) 
tensor product probability measure
\begin{equation}\label{eq:rhoprod}
	\varrho = \varrho^{(\infty)} = \bigotimes_{j\in \IN} \varrho^{(1)}
\end{equation}
on the parameter domain $\IU = [-1,1]^\IN$. For an integrability index $p\in (0,\infty]$
we define the Lebesgue spaces $L^p_\varrho(\IU;\IR)$
in the usual fashion as (equivalence classes of)
strongly $\varrho$-measurable maps $f:\IU\to \IR$
for which $\| f \|_{L^p_\varrho(\IU)} < \infty$
where
\begin{equation}\label{eq:defLp}
	\| f \|_{L^p_\varrho(\IU;\IR)} 
	:=
	\left\{ \begin{array}{lr} \displaystyle \left( \int_{\IU} |f(\bsy)|^p d\varrho(\bsy) \right)^{1/p} 
		& 1\leq p < \infty
		\\
		{\rm ess sup}_{\bsy\in \IU} |f(\bsy)| & p=\infty
	\end{array}
	\right.
\end{equation}
See \cite{evans2022partial, adams} for more details on these various function spaces.
\subsection{Generalized polynomial chaos approximation}
\label{sec:gPC}
We recap several (known) results on rates of $n$-term approximation rates of
parametric-holomorphic, scalar functions.
Only particular cases of more generally valid approximation rate bounds 
in e.g. \cite{DSZ17}
are cited which have relevance to the present derivations; 
in particular, we limit ourselves to so-called 
\emph{Taylor-} \cite{ChSJZDlHd19} and \emph{Chebyshev-gPC expansions},
considered in \cite{ChSJZDlHd19,BCDS17_2452,RS16_1344,DSZ17}, 
for example.
%

%
\subsubsection{Holomorphy}
\label{sec:Hol}
%
Approximation rate bounds of parametric maps 
$f:\IU\to \IC$
will rely on \emph{holomorphic dependence} of $f$ on
its argument $z\in \IU$, as specified in the following
\begin{definition}\label{def:Hol}
	For $d\in \IN$ or $d=\infty$,
	and for an open set $\cO\subseteq \IC^d$,
        resp. $\cO \subset \IC^\IN$ which is open w.r. to the product topology on $\IC^\IN$,
	a map $f:\cO\to \IC$ is holomorphic in $\cO$
	iff it is continuous and if it is separately
	holomorphic with respect to each argument $z_j$, i.e.
	for any fixed $z\in \cO$, and any $j\in \{1,...,d\}$,
	the limit
	$$
	\lim_{\IC\ni h \to 0} \frac{f(x+he_j)-f(z)}{h} 
	$$
	exists in $\IC$. 
        For a map $f:\IU\to \IR$ and an open set $\cO$ with $\IU \subseteq \cO\subseteq \IC^d$,
	we say that \emph{$f$ admits a holomorphic extension to $\cO$}
	if there exists
	a holomorphic $\tilde{f}:\cO\to \IC$ such that $\tilde{f}|_{\IU} = f$ on $\IU$.
\end{definition}

We remark that parametric solution families of linear, elliptic or parabolic PDEs
afford parametric holomorphic solutions and maps $f$ as considered here
are obtained by so-called \emph{quantities of interest} (``QoIs'') given as, eg.,
linear functionals of parametric solutions \cite{ChSJZDlHd19}.
In these references, approximation rate bounds free from the CoD (in the case $d=\infty$)
were obtained from $p$-summabilities of (sequences of) gPC expansion coefficients
via Stechkin's lemma.
Such summabilities were shown in \cite{CDS11,ChSJZDlHd19} to follow from
order-explicit bounds on the gPC coefficients which, in turn, can be obtained via
\emph{quantified holomorphy} on suitable closed Bernstein-Ellipses $\cE_\rho\subseteq \IC$
with foci at $z = \pm 1$ and semiaxis-sum $\rho>1$, 
i.e.
$$
\cE_\rho = \left\{ (z+z^{-1})/2 \mid z\in \IC, \; 1\leq |z| \leq \rho \right\}\;.
$$
Bounds on Chebyshev-gPC coefficients in expansions of $f(z)$ require multivariate
versions of $\cE_\rho$, obtained as cartesian products of coordinate-wise
copies of $\cE_{\rho_j}$ in $z_j\in \IC$:
for any $d>1$, we set
$\bsrho = \prod_{1\leq j \leq d} \rho_j$ (countable product in case that $d=\infty$),
and define the closed poly-ellipse
$$
\cE_\bsrho := \cE_{\rho_1} \times \cE_{\rho_2} \times \hdots  \times \cE_{\rho_d} \; \subseteq \IC^d \;.
$$
The set of all $f:\cE_\bsrho \to \IC$ which are
holomorphic in $\cE_\bsrho$ will be denoted by $\cH(\bsrho)$, i.e.
\begin{equation}\label{eq:Holrho}
	\cH(\bsrho) := \{ f:\cE_\bsrho \to \IC \mid f\;\mbox{holomorphic in}\; \cE_\bsrho, 
	\| f \|_{L^\infty(\cE_\bsrho)} < \infty \} \;.
\end{equation}
For $d=\infty$,
bounds on gPC-coefficient sequences of $f$ require holomorphy on
subsets of $\IC^\IN$ which are unions of all poly-ellipses $\cE_\bsrho$
for which $\bsrho$ is $(\bsb,\epsilon)$-admissible:
for $0<p<1$ and a sequence
$\bsb = (b_j)_{j\geq 1} \in (0,\infty)^\IN\cap \ell^{p}(\IN)$,
we define the cylindric 
\emph{ellipsoidal holomorphy domain}
\begin{equation}\label{eq:HolSet}
	\cE(\bsb,\epsilon) := \bigcup \left\{ \cE_\bsrho \mid \rho_j \geq 1, 
	\sum_{j\geq 1} b_j \left( \frac{\rho_j + 1/\rho_j}{2} - 1 \right) \leq \epsilon \right\}
	\;.
\end{equation}
Similarly, 
given polyradius $\bsrho\in [1,\infty)^\IN$, 
and polydiscs 
$B_\bsrho = B_{\rho_1} \times B_{\rho_2} \times ... \subseteq \IC^\IN$,
with 
$B_r := \{z\in \IC | |z| \leq r \}$,
define \emph{polydisc holomorphy domains}
\begin{equation}\label{eq:HolDisc}
	\mathcal{D}(\bsb,\epsilon) 
	:= 
	\bigcup \left\{ B_\bsrho | \; \rho_j\geq 1,\; \sum_{j\geq 1} b_j(\rho_j-1) \leq \epsilon \right\}
	\;.
\end{equation}
The conditions on the poly-radii $\bsrho$ in  \cref{eq:HolSet}, \cref{eq:HolDisc} are referred to
below and the references cited here as \emph{$(\bsb,\epsilon)$-admissibility}.
In \cref{eq:HolSet}, \cref{eq:HolDisc}
the dependence on $\epsilon>0$ could be absorbed into the sequence $\bsb$; 
however, it will be convenient in applications to keep track 
of $\bsb$ and $\epsilon$ separately.
We also consider the approximation of infinite dimensional holomorphic functions, in the following sense.

\begin{definition}\label{def:bepsHol}
	Let $p \in (0,1)$ and let $\bsb \in \ell^p(\IN)$ be a nonnegative
	sequence. 
	A function $u:\IU \to \R$ is called
	\emph{$(\bsb, p)$-holomorphic} if it can be expressed as
	\begin{equation} \label{eq:Parfct}
		u(\bsy)= U \Big(\sum_{j\in \IN} y_j\psi_j\Big)\qquad\forall\bsy\in\IU,
	\end{equation}
	where for some Banach space $X$, the map $U:X\to\R$ and the sequence
	$(\psi_j)_{j\in\IN}\subseteq X$ satisfy the following: 
        it holds
	$\| \psi_j \|_X \le b_j$
	for all $j\in\IN$, there exists
	an open subset $O\subseteq X_\C$ of the complexified Banach space $X_\C$, 
        such that
	$\set{\sum_{j\in \IN}y_j\psi_j:\bsy\in \IU}\subseteq O$,
	and such that $U:O\to\C$ is holomorphic.
\end{definition}

	
	
%
%
Countably-parametric expressions such as \cref{eq:Parfct} which arise as functionals $G(.)$ 
of a 
parametric family of PDE solutions as in \cref{eq:Diff}, \cref{eq:DiffPar} 
are known to allow \emph{generalized polynomial chaos} (gPC) expansions 
(see, e.g., \cite{BCDS17_2452,Z18_2760} and references there).
For such gPC expansions, 
a range of univariate orthogonal polynomial systems in $[-1,1]$ 
can be considered (e.g. \cite{Z18_2760} and the references there). 
We next recap two of these which are particular amenable to representations by PQCs:
first, 
the so-called Taylor gPC comprising multivariate expansions 
into monomials 
$\{ \bsy^\bsnu: \bsnu\in \cF \}$ (see \cite{CDS11,ChSJZDlHd19}), 
and
second, the so-called Chebyshev-gPC expansions which are based on the 
system $\{ T_j \}_{j\geq 0}$ of Chebyshev-polynomials of the first kind, 
albeit with a non-classical (but crucial for PQC construction) 
normalization in \cref{sec:TschgPC} below.
\subsubsection{Taylor-gPC expansions}
\label{sec:TaygPC}
Taylor-gPC expansions for parametric, holomorphic solution families were used 
e.g. in \cite{CCDS13_479,CCS15} and in \cite{ChSJZDlHd19}, 
where neural network approximations of such expansions were constructed. 
We build PQC surrogates of Taylor-gPC expansions, using the results 
in \cite[App.~B.1]{yu2023provable}. 
Taylor-gPC expansions are formal expressions of the type
\begin{equation}\label{eq:TaylorgPC}
	u(\bsy) = \sum_{\bsnu\in \cF} t_\bsnu \bsy^\bsnu \;, \quad \bsy \in \cU \;,
\end{equation}
where the 
Taylor gPC coefficients $t_\bsnu$ of $u:\bsy \to \IR$ 
are given by
$$
t_\bsnu := \frac{1}{\bsnu!} (\partial^\bsnu_\bsy u)(\bsy) \mid_{\bsy = 0} \in \IR
\;.
$$
References \cite[Prop.~2.3]{ChSJZDlHd19} and \cite{CDS11,CCS13} 
state that for parametric maps $u:U\to \IR$ which 
admit a holomorphic extension to the 
polydisc $\mathcal{D}(\bsrho,\epsilon)$ defined in \cref{eq:HolDisc}
can be represented in the form \cref{eq:TaylorgPC}
with the Taylor-gPC expansion \cref{eq:TaylorgPC} 
converging unconditionally for all $\bsy \in U = [-1,1]^\IN$.
\subsubsection{Chebyshev-gPC expansions}
\label{sec:TschgPC}
We denote by $\{ T_j \}_{j\geq 0}$ the sequence of Chebyshev-polynomials
in $[-1,1]$,
assumed orthonormalized with respect to the probability Chebyshev measure, $\varrho^{(1)}$.
This definition differs from that of the classical Chebyshev-polynomials
$\tilde{T}_j$ which are normalized according to $\tilde{T}_k(1)=1$ for $k=0,1,2,...$. 
There holds $\sup_{|x| \leq 1} |\tilde{T}_k(x)| = 1$ 
and 
$$
T_0 = \pi^{-1/2} \tilde{T}_0\;,\quad T_k = (2/\pi)^{1/2} \tilde{T}_k\;,\; k=1,2,... \;.
$$
In particular also for all $k=0,1,2,...$
$$
\sup_{|x|\leq 1} |T_k(x)| \leq (2/\pi)^{1/2} < 1 
\;.
$$
The $T_j$ form an orthonormal basis (ONB) of the
Hilbert space $L^2_{\varrho^{(1)}}([-1,1];\IR)$. For $d\geq 1$ and for $\bsnu\in \cF$,
we define the tensor-product Chebyshev-polynomials
\begin{equation}\label{eq:Tsch}
	T_\bsnu(\bsy) := \prod_{j\geq 1} T_{\nu_j}(y_j) \;, 
	\quad 
	\bsy = (y_j)_{j\geq 1} \in \IU \;.
\end{equation}
Observe that due to $\bsnu\in \cF$ 
the products in \cref{eq:Tsch} have only a finite number $|\bsnu|_0$ 
of nontrivial terms, and $\forall \bsnu\in \cF$ we have
$$
\sup_{\bsy\in \cU} | T_\bsnu(\bsy) | \leq (2/\pi)^{|\bsnu|_0/2} < 1 \;.
$$
The system $\{ T_\bsnu \}_{\bsnu\in \cF}$
is an ONB of $L^2_\varrho(\IU;\IR)$.
i.e., every $f\in L^2_\varrho(\IU;\IR)$ admits the $L^2_\varrho(\IU;\IR)$-convergent
expansion
\begin{equation}\label{eq:L2Series}
	f = \sum_{\bsnu\in \cF} c_\bsnu T_\bsnu \;,
	\quad 
	c_\bsnu := \int_{\IU} f(\bsy) T_\bsnu(\bsy) d\varrho(\bsy) \;.
\end{equation}
The correspondence
\begin{align}
\label{eq:L2seq}
\mathcal I : L^2_\varrho(\IU) & \to \ell^2(\cF) \\
f & \mapsto \bsc = (c_\bsnu)_{\bsnu\in \cF}
\end{align}
induced by \cref{eq:L2Series} is an isometric isomorphism. 

For numerical approximation of e.g. $u(\bsy)$ in \cref{eq:TaylorgPC},
the formal expansions \cref{eq:TaylorgPC}, \cref{eq:L2Series} 
will be truncated with summation over a finite index set $S\subseteq \cF$. For a finite index set $\emptyset \ne S\subseteq\cF$, 
define the finite-dimensional polynomial space
\begin{equation}\label{eq:defPS}
	\cP_S = {\rm span} \{ T_\bsnu \mid \bsnu\in S \} \subseteq L^2_\varrho(\IU)
	\;.
\end{equation}
A \emph{best $n$-term $L^2_\varrho$ Chebyshev-approximation of $f \in L^2_\varrho(\IU;\IR)$} 
is given by
$$
f_n := {\rm arg min}\left\{ \| f-g \|_{L^2_\varrho(\IU)} : 
g\in \cP_S, S\subseteq \cF, |S|\leq n \right\} \;.
$$
An explicit best $n$-term approximation $f_n$ is given by
$$
f_n = \sum_{\bsnu\in S^*} c_\bsnu T_\bsnu 
$$
where $S^*\subseteq \cF$ is a subset with $|S^*| \leq n$ that contains
multi-indices of $n$ largest coefficients $c_\bsnu$ 
of the Chebyshev-gPC expansion of $f$ in \cref{eq:L2Series}.
It holds
\begin{equation}\label{eq:L2Error}
	\| f - f_n \|_{L^2_\varrho(\IU)}^2 
	=
	\sum_{\cF\backslash S^*} |c_\bsnu|^2
	=
	\| \bsc \|_{\ell^2(\cF\backslash S^*)}^2
	\;.
\end{equation}
Apparently, the polynomial space $\cP_S$ in \cref{eq:defPS} 
depends on the chosen basis to define the span.
When $S$ is a d.c. set, however, $\cP_S$ is independent of the choice of univariate basis.
\begin{lemma}
	\label{lem:cPSBas}
	Assume the index set $\emptyset \ne S\subseteq\cF$ is finite and d.c. Then, 
	for any univariate basis $\{ p_j \}_{j\geq 0}$ of polynomials of precise\footnote{I.e.,
		$p_j(x) = c_jx^j + ...$ with $c_j\ne 0$.} 
	degree $j \geq 0$, there holds
	$$
	\cP_S 
	= {\rm span}\{ \bsy^\bsnu \mid \bsnu\in S \} 
	= {\rm span}\{ P_\bsnu(\bsy) : \bsnu \in S \}.
	$$
Here, for $\bsnu\in \cF$, $P_\bsnu(\bsy)  = p_{\nu_1}(y_1) p_{\nu_2}(y_2)...$
with the dots indicating the product of all nonzero entries $\nu_j$ of $\bsnu$
(noting that $\tilde{T}_0 \equiv 1$ for Chebyshev-polynomials).
\end{lemma} 
For $0<p\leq \infty$, we denote by $\ell^p(\cF)$ the sets of
sequences $\bsc  = (c_\bsnu )_{\bsnu\in \cF}$ such that
\begin{align*}
	\| \bsc \|_{\ell^p(\cF)} 
	&:= 
	\left( \sum_{\bsnu\in \cF} | c_\bsnu |^p \right)^{1/p}
	<\infty \quad \text{for} \quad 0<p<\infty\;, \\
	\| \bsc \|_{\ell^\infty(\cF)} 
	&:= \sup_{\bsnu \in \cF} |c_\bsnu| < \infty
	\quad \text{for} \quad p = \infty\;.
\end{align*}
We do not distinguish between real-valued and complex-valued sequences,
with $|\cdot|$ denoting absolute value and complex modulus, respectively.
%
\begin{definition}\label{def:Bstnterm}
	For $\Lambda \subseteq \cF$, $0<p\leq \infty$, $\bsc \in \ell^p(\Lambda)$ 
	and for $n\in \IN_0$, with $n\leq \#(\Lambda)$, 
	the 
	\emph{$\ell^p$-norm best $n$-term approximation error of $\bsc$} 
	is given by
	$$
	\sigma_n(\bsc)_p := \min\{ \| \bsc - \tilde{\bsc} \|_p : \tilde{\bsc}\in \ell^p(\Lambda),
	|{\rm supp}(\tilde{\bsc})| \leq n \}
	\;.
	$$
\end{definition}
For $f\in L^2_\varrho(\IU)$ as in \cref{eq:L2Series}, 
for $p=2$ we obtain
$$
\sigma_n(\bsc)_2 = \| f - f_n \|_{L^2_\varrho(\IU)} \;,
$$
where $f_n$ is a best $n$-term Chebyshev-polynomial approximation for $f\in L^2_\varrho(\IU)$.
%

\subsubsection{$n$-term Taylor-gPC approximation}
\label{sec:BestnTay}
The following result on the convergence rate of $n$-term truncations 
of Taylor-gPC expansions of $(\bsb,\epsilon)$-holomorphic, parametric 
functions is proved as \cite[Thm.~2.7]{ChSJZDlHd19}.
\begin{proposition}\label{prop:TaygPC}
	Let $u:\cU\to \IR$ be $(\bsb,\epsilon)$-holomorphic 
	for some sequence $\bsb \in (0,\infty)^\IN \cap \ell^p(\IN)$ for some $0<p\leq 1$. 
	Then the sequence $( t_\bsnu )_{\bsnu\in \cF} \in \ell^p(\cF)$. There exists a constant $C>0$ and a sequence $\{ \Lambda_n \}_{n\geq 1}$
	of nested, d.c. index sets $\Lambda_n\subseteq \cF$ such that for all $n\in \IN$ 
	holds $|\Lambda_n| \leq  n$ ,
	and the parametric function $u$ in \cref{eq:TaylorgPC} 
	satisfies  
	\begin{equation}\label{eq:TaygPC1}
		\sup_{\bsy \in \cU} \left| u(\bsy) - \sum_{\bsnu\in \Lambda_n} t_\bsnu \bsy^\bsnu \right| 
		\leq 
		\sum_{\bsnu\in \cF\backslash \Lambda_n} |t_\bsnu| 
		\leq C n^{-(1/p-1)} \;,
	\end{equation}
	and
	\begin{equation}\label{eq:TaygPC2}
		\max_{\bsnu \in \Lambda_n} |\bsnu|_1 \leq C (1+\log(n)) \;.
	\end{equation}
	Furthermore, if $\boldsymbol{e}_j\in \Lambda_n$ for some $j\in \mathbb N$, then also
	$\boldsymbol{e}_i\in \Lambda_n$ for all $i=1,2,...,j$ (this property of the 
	index sets $\Lambda_n$ is also referred to as ``anchored'').
\end{proposition}  
\subsubsection{$n$-term Chebyshev-gPC approximation}
\label{sec:Bestnterm}
We are primarily interested in convergence rates of finite-parametric truncations
of the Chebyshev-gPC expansions whose coefficient sequences $\bsc\in \ell^p(\cF)$
for some $0<p<2$.
The following approximation rate bounds of best $n$-term Chebyshev-polynomial
approximations of certain \emph{holomorphic} $f\in L^2_\varrho(\IU)$ hold true.
Each of these results will be shown to allow a quantum analogue, via
suitable quantum circuits emulating Chebyshev-polynomials through the LCU
technique of \cite{Childs2012lcu}. The next theorem is a special case of \cite[Thm.~2.2.10]{Z18_2760}. 
Similar and related statements can be found in the earlier publications \cite{CDS10,CDS11,CCS13,CCS15}. 
\begin{theorem}[infinite-dimensional case]\label{thm:AlgInfDim}
	Assume $d=\infty$ and that
	$f\in \cH(\bsb,p,\epsilon)$ for some
	$0<p\leq 1$, $\epsilon>0$ and
	$\bsb\in (0,\infty)^\IN\cap \ell^p(\IN)$. Then, for every $n\in \IN$
	there exists a dc index set
	$\Lambda_n\subseteq \cF$
	of cardinality $n$ such that
	\begin{equation}\label{eq:LnlogBd}
		\max_{\bsnu\in \Lambda_n} | \bsnu |_1 \leq C(1+\log(n)) 
	\end{equation}
	and with the truncated Chebyshev-expansion 
	$f^{\operatorname{Cheb}}_{\Lambda_n} := \sum_{\bsnu\in \Lambda_n} c_\bsnu T_\bsnu$
	it holds
	\begin{equation}\label{eq:InfDimErr}
		\| f - f^{\operatorname{Cheb}}_{\Lambda_n} \|_{L^2_\varrho(\IU)}
		\leq 
		C(f) n^{-1/p+1/2} ,
		\quad 
		\| f - f^{\operatorname{Cheb}}_{\Lambda_n} \|_{L^\infty(\IU)}
		\leq 
		C(f) n^{-1/p+1} ,
	\end{equation}
	for some constant $C(f)$.
\end{theorem}
	
	
\begin{proof}
	Throughout this proof fix $\tau>0$, such that for all $\bsnu\in\cF$ holds
	\begin{equation}\label{eq:ChebBound}
		\| T_\bsnu \|_{L^\infty(\IU)}\le \prod_{i\in\IN}(1+\nu_i)^\tau=:w_\bsnu\qquad\forall \bsnu\in\cF
	\end{equation}
	See for instance \cite[Appendix B.2.1]{Z18_2760}. According to \cite[Thm.~2.2.10 (i)-(iii)]{Z18_2760} (in the Chebyshev-polynomial case, i.e., with $\alpha=\beta=-1/2$), 
	there exists a sequence $(a_\bsnu)_{\bsnu\in\cF}$ 
	of nonnegative real numbers with the following properties: it holds that
	\begin{equation}\label{eq:tau}
		w_\bsnu |c_\bsnu|\le a_\bsnu \qquad\forall \bsnu\in\cF,
	\end{equation}
	$(a_\bsnu)_{\bsnu\in\cF}\in\ell^p(\cF)$, and $(a_\bsnu)_{\bsnu\in\cF}$ is monotonically decreasing 
	in the sense that $a_\bsnu\ge a_\bsmu$ whenever $\bsnu\le\bsmu$.
	Let $(\bsnu_i)_{i\in\IN}$ be an arbitrary enumeration of $\cF$ such that $(a_{\bsnu_i})_{i\in\IN}$ 
	is monotonically decreasing and $\bsnu_i\le\bsmu_j$ implies $i\le j$. 
	Such an enumeration exists due to the above stated properties. Then, for each $n\in\IN$
	\begin{equation*}
		\Lambda_n:=\set{\bsnu_i:i\le n}\subseteq\cF
	\end{equation*}
	defines a d.c. set of cardinality $n$ such that $(a_{\bsnu_i})_{i=1}^n$ 
	corresponds to $n$ largest numbers of the sequence $(a_\bsnu)_{\bsnu\in\cF}$. It remains to verify the error bounds \cref{eq:InfDimErr}, 
	which are a consequence of Stechkin's lemma \cite{stechkin}. 
	We start with $L_\varrho^2(\IU)$. 
	Using Parseval's identity
	\begin{align}\label{eq:L2rho}
		\|f-f^{\operatorname{Cheb}}_{\Lambda_n}\|_{L^2_\varrho(\IU)}^2
		=
		\| \sum_{\bsnu\in\cF} c_\bsnu T_\bsnu - \sum_{\bsnu\in\Lambda_n} c_\bsnu T_\bsnu \|_{L^2_\varrho(\IU)}^2
		=\sum_{\bsnu\in\cF\backslash\Lambda_n} c_\bsnu^2 \le \sum_{i>n} a_{\bsnu_i}^2.
	\end{align}
	As the sequence  
	$(a_{\bsnu_i})\in\ell^p(\IN)$ 
	is monotonically decreasing, 
	it holds for all $j\in\IN$
	\begin{equation}\label{eq:ai}
		a_{\bsnu_j}^p 
		\le \frac{1}{j}\sum_{i=1}^j a_{\bsnu_i}^p 
		\le \frac{1}{j}\sum_{i\in\IN}a_{\bsnu_i}^p\quad\text{so that}\quad
		a_{\bsnu_j}\le C j^{-1/p},
	\end{equation}
	where 
	$C:=\|(a_\bsnu)_{\bsnu\in\cF}\|_{\ell^p}<\infty$. 
	Therefore
	\begin{equation*}
		\sum_{i>n}a_{\bsnu_i}^2\le C^2 \sum_{i>n}i^{-2/p}\lesssim n^{-2/p+1},
	\end{equation*}
	which, together with \cref{eq:L2rho}, 
	gives the desired bound in $L^2_\varrho(\IU)$. 
	For the $L^\infty(\IU)$ bound we proceed in a similar fashion. 
	Using \cref{eq:ChebBound} and \cref{eq:tau}, 
	\begin{align*}
		\|f-f^{\operatorname{Cheb}}_{\Lambda_n}\|_{L^\infty(\IU)}
		&=
		\|\sum_{\bsnu\in\cF}c_\bsnu T_\bsnu-\sum_{\bsnu\in\Lambda_n}c_\bsnu T_\bsnu \|_{L^\infty(\IU)}
		\\
		&\le \sum_{\bsnu\in\cF\backslash\Lambda_n}|c_\bsnu|\|T_\bsnu\|_{L^\infty(\IU)} \le \sum_{i>n}a_{\bsnu_i}.
	\end{align*}
	Together with \cref{eq:ai} this gives
	\begin{equation*}
		\|f-f^{\operatorname{Cheb}}_{\Lambda_n}\|_{L^\infty(\IU)} \le C \sum_{i>n} i^{-1/p}\lesssim n^{-1/p+1}
	\end{equation*}
	and completes the proof.
\end{proof}
The point of the preceding result is that, formally,
the function $f$ may depend on a countable number of coordinates
$\bsy = (y_j)_{j\geq 1}$.
The  approximation rate bounds
\cref{eq:InfDimErr} are independent of the number of
``active'' coordinates contributing to the $n$-term approximants $f_{\Lambda_i}$.
The approximation rate bounds \cref{eq:InfDimErr} only depend on
the sparsity in the coefficient sequences $\bsc$, 
as expressed through the summability exponent $p\in (0,1)$.
In this sense, the convergence rate bounds 
\cref{eq:InfDimErr} do not incur the CoD. 
Theorem~\cref{thm:AlgInfDim} is formulated for $d=\infty$.
It is well-known and classical, that for $1\leq d <\infty$,
for $f\in \cH(\bsrho)$, suitably truncated $n$-term 
Chebyshev-expansions converge exponentially \cite{HTW17,OSZ2022}.
\begin{theorem}\label{thm:ExpFinDim}
	{(finite-dimensional case)}
	Assume that $1\leq d < \infty$ and that $\IU = [-1,1]^d$.
	Let $f\in \cH(\bsrho)$ for some $\bsrho \in (1,\infty)^d$
	and
	let $\bsc = (c_\bsnu)_{\bsnu\in \IN_0^d} \in \ell^2(\IN_0^d)$
	denote the sequence of Chebyshev-coefficients of $f$. 

        Then, 
		there exists a constant $C>0$ (depending on $d$ and on $f$) 
		and a constant $\gamma>0$ (depending on $\bsrho$) 
		such that,
		for every $n\in \IN$, there exists 
		a dc set $\Lambda_n \subseteq \IN_0^d$ with $|\Lambda_n|\leq n$ 
		such that there holds 
		\begin{equation}\label{eq:FinDimErr}
			\| f - f^{\operatorname{Cheb}}_{\Lambda_n} \|_{L^2_\varrho(\IU)} 
			\leq 
			\| f - f^{\operatorname{Cheb}}_{\Lambda_n} \|_{L^\infty_\varrho(\IU)} 
			\leq 
			C\exp(- \gamma n^{1/d}) 
		\end{equation}
		where the norms are as in \cref{eq:defLp} 
		and the polynomial approximation is 
		$f^{\operatorname{Cheb}}_\Lambda := \sum_{\bsnu\in \Lambda} c_\bsnu T_\bsnu$. The index sets $\Lambda_n \subseteq \IN_0^d$ 
		in \cref{eq:FinDimErr}
		can be chosen as product sets
		\begin{equation}\label{eq:LtTP}
			\Lambda_n 
			:= 
			\{ \bsnu\in \IN_0^d : | \bsnu |_\infty \leq k \} 
			= 
			\{ 0,1,2,...,k\}^d 
		\end{equation}
		with $n = \#(\Lambda_n) = \mathcal O (k^d)$. 
		The index sets $\Lambda_n$ are nested, and are d.c.
		The approximation $f^{\operatorname{Cheb}}_\Lambda$ is the tensor product projection of 
		$f$ onto $\IQ_k^d = \IP_k^{\otimes d}$.
\end{theorem}
\subsection{Quantum circuit design via LCUs}
\label{sec:Qc}
We review the PQC design via LCUs, 
as developed in \cite{Childs2012lcu,gilyen2019qsvt}.
In \cref{subsection:qsp}, 
we discuss the quantum signal processing (QSP) algorithm. 
In \cref{lcu-sec}  , 
we present the linear combination of unitaries (LCU) technique,
with particular attention on the emulation of Chebyshev-polynomials. 
\subsubsection{Quantum Signal Processing}
\label{subsection:qsp} 
Quantum signal processing (QSP) represents a real, scalar polynomial of
one variable and of degree $d$ using a product of 
unitary matrices of size $2 \times 2$ representing suitable single qubit rotation and single qubit phase gates that are parameterized by $d+1$ real numbers 
$\{\psi_j\}_{j=0}^d\in [-\pi/2,\pi/2]^{d+1}$ 
which are referred to as 
``phase factors'', or ``phase angles'', see also \cite{dong2022infiniteqsp}; 
we adopt the latter terminology. More formally, 
let $x \in [-1, 1]$ be a scalar with a one-qubit encoding:
\begin{equation*}\label{equation:encoding}
	W(x) := 
	\begin{pmatrix}
		x & i\sqrt{1 - x^2} \\
		i\sqrt{1 - x^2} & x
	\end{pmatrix}
	:= e^{i \arccos(x) \sigma_X},
	\quad \quad
	\theta \in [0,\pi].
\end{equation*}
For a vector of phase-angles
$\boldsymbol \varphi = (\varphi_0, \varphi_1, \ldots, \varphi_\ell) \in \mathbb R^{\ell+1}$, 
consider the following sequence of quantum gates:
\begin{equation}\label{equation:qsp-ansatz}
	V_{\boldsymbol \varphi}(x) = e^{i\varphi_0 \sigma_Z} W(x) e^{i\varphi_1 \sigma_Z} W(x) e^{i\varphi_2 \sigma_Z} \ldots W(x) e^{i\varphi_\ell \sigma_Z}.
\end{equation}
The main result that we now discuss is 
which class of polynomial functions can be encoded using 
the sequence of gates in \cref{equation:qsp-ansatz}.
\begin{proposition}\label{thm:qsp-complex-poly} \cite[Theorem 3]{gilyen2019qsvt}
	Let $\ell \in \mathbb{N}$.
	There exists 
	$\boldsymbol \varphi = (\varphi_0, \varphi_1, \ldots, \varphi_\ell) \in \mathbb{R}^{\ell+1}$ 
	such that for all $x\in[-1,1]$
	\begin{equation*}\label{equation:qsp-result}
		V_{\boldsymbol{\varphi}}(x) 
		= e^{i\varphi_0\sigma_Z} \prod_{j=1}^{\ell} \left( W(x) e^{i\varphi_j\sigma_Z} \right)
		=
		\begin{pmatrix}
			p(x) & iq(x)\sqrt{1-x^2} \\
			iq^*(x)\sqrt{1-x^2} & p^*(x)
		\end{pmatrix}
	\end{equation*}
	if and only if $p, q \in \mathbb{C}[x]$ such that:
	\begin{enumerate}
		\item[(i)] $\operatorname{deg}(p(x)) \leq \ell$ and $\operatorname{deg}(q(x)) \leq \ell - 1,$
		\item[(ii)] $p(x)$ has parity $(\ell \mod 2)$ and $q(x)$ has parity $(\ell - 1 \mod 2)$,
		\item[(iii)] For all  $x \in [-1, 1]$, we have $|p(x)|^2 + (1 - x^2) |q(x)|^2 = 1$.
	\end{enumerate}
\end{proposition}

\begin{remark}
	A polynomial has parity $0$ if all coefficients corresponding to odd powers of $x$ are $0$. 
	Similarly, a polynomial has parity $1$ if all coefficients corresponding to even powers of $x$ are $0$. 
\end{remark}

Quantum Signal Processing (QSP) forms the foundational basis for the quantum circuit constructions presented in this work, owing to its intrinsic ability to implement Chebyshev polynomial transformations efficiently within quantum algorithms.

\begin{lemma}\label{lem:qsp-1d-chebyshev} 
	Let $T_k \in \mathbb{R}[x]$ be the degree-$k$ Chebyshev polynomial of the first kind. 
	Let further 
	$\boldsymbol \varphi \in \mathbb{R}^{k+1}$ such that $\varphi_i = 0$
	for all $i = 0, \cdots, k$.  
	For this specific choice $\boldsymbol \varphi$, 
	there exists an unparameterized quantum circuit, $U_k(x)$, 
	such that 
	\begin{equation}
		\bra{0} U_k(x) \ket{0}  = T_k(x)
	\end{equation}
	for each $k \in \mathbb N \cup \{0\}$. 
	The PQC $U_k$ has width $1$, depth $2k+1$, and $k+1$ (predetermined) phase angle parameters.
\end{lemma}
\begin{remark}\label{rmk:TschQSVD}
	Reference~\cite{gilyen2019qsvt} establishes a result closely related to Lemma~\cref{lem:qsp-1d-chebyshev}, albeit with minor differences. These discrepancies stem from the use of a slightly different quantum signal processing circuit configuration than the one adopted in our work. A proof of Lemma~\cref{lem:qsp-1d-chebyshev} that aligns with our conventions can be found, for example, in~\cite[Lem. 3.1]{aftab2024korobov}.
\end{remark}
\begin{remark}\label{rmk:InfQSP}
	The vector of phase angles in the QSP representation is, in general, 
	not unique, as discussed in the univariate case in \cite{Dong2024infinitequantum}.
	As shown in \cite[Thm.~3]{Dong2024infinitequantum}, 
	under additional conditions a unique selection, 
	dubbed ``reduced phase factors'', 
	is shown to exist, and to admit efficient, stable numerical approximation 
	by a fixed point iteration.
	Furthermore, 
	\cite[Thm.~4]{Dong2024infinitequantum} shows that the
	entries of these reduced phase angle vectors decay, 
	as the polynomial degree of the Chebyshev-polynomial tends to infinity,
	\emph{in-step with the Chebyshev-coefficients of the expansion in Chebyshev-polynomials}.
	A tensorization argument could imply a corresponding result
	also for the presently considered, multivariate Chebyshev-expansions. 
	The multivariate arrays of reduced phase-angles corresponding to Chebyshev-gPC expansions
	should then inherit the summability properties of the Chebyshev-gPC coefficient sequences.
	Sparsity of the reduced phase-angle arrays can be expected if 
	the gPC coefficient sequences are.
\end{remark}

\subsubsection{Linear Combination of Unitaries}\label{lcu-sec}
The Linear Combination of Unitaries (LCU) method, originally proposed in~\cite{Childs2012lcu}, offers a systematic approach for implementing linear combinations of unitary operators within a quantum computational framework. Since its introduction, the LCU technique has become a fundamental component in the design of a wide array of quantum algorithms. Specifically, given integers \( k, T \in \mathbb{N} \), real coefficients \( a_1, \ldots, a_T \in \mathbb{R} \), and unitary operators \( U_1, \ldots, U_T \), the LCU method enables the realization of the operator
\[
U = \sum_{j=1}^{T} a_j U_j.
\]
This construction relies on the capability to implement two specific unitary operators, which are essential to the execution of the LCU algorithm. First, the algorithm assumes that the state preparation oracles $F_l$ and $F_r$ can be implemented:    
	\begin{align}\label{F-LCU}
		F\ket{0} = \frac{1}{\sqrt{\| \boldsymbol{a} \|_1 }} \sum_{j=1}^{T} \sqrt{a_{j}} \ket{j}.
	\end{align} Second, it assumes that the following two-qubit controlled gate can be implemented:
	\begin{equation}\label{controlled-LCU}
		U_c = \sum_{j=1}^{T} U_j \otimes \ket{j}\bra{j}
	\end{equation} 
Given access to the necessary unitary operators, the LCU algorithm enables the implementation of the target unitary operator \( W_{\operatorname{LCU}} = (I \otimes F^{\dagger}) U_c (I \otimes F) \). Let \( \ket{\psi} \) represent the input quantum state, and let \( \ket{0} \) be an ancilla qubit. The application of \( W_{\operatorname{LCU}} \) to the quantum state is described by the following expression:
\[
W_{\operatorname{LCU}} \ket{\psi} \ket{0} = \frac{1}{\|\boldsymbol{a}\|_1} \left( \sum_{j=1}^{T} a_j U_j \right) \ket{\psi} \ket{0} + \ket{\perp},
\]
where \( \ket{\perp} \) denotes a potentially non-normalized state that satisfies the condition \( (I \otimes \ket{0}\bra{0}) \ket{\perp} = 0 \). Our primary objective is to compute the expected value \( \bra{\psi} W_{\operatorname{LCU}} \ket{\psi} \) for a given quantum state \( \ket{\psi} \). This expectation can be efficiently calculated using the Hadamard test \cite{1998clevequantumalgorithmsrevisited}. The test takes the quantum state \( \ket{\psi} \) and the operator \( W_{\operatorname{LCU}} \) as input, and produces a random variable from the measurement of the ancilla qubit, from which the desired expected value is derived. 
	Finally, we have the following expression for the output of the computation:
\begin{equation}\label{had-prob}
\bra{0} \bra{\psi} W_{\operatorname{LCU}} \ket{\psi} \ket{0} = \frac{1}{\|\boldsymbol{a}\|_1} \bra{\psi} \left( \sum_{j=1}^{T} a_j U_j \right) \ket{\psi}.
\end{equation}	
	
\begin{remark}We note that \cref{had-prob} implies that the success probability decays as \( \mathcal O ( 1 / \| \boldsymbol{a} \|_1^2 ) \). In particular, if \( \| \boldsymbol{a} \|_1 = \Omega(c^n) \) for some $c > 1$, then the success probability decays exponentially. In the worst case, to prevent this exponential decay, it is necessary to employ robust oblivious amplitude amplification \cite{childs2015lcutaylor} to boost the success probability to \( 1 - \mathcal O (\delta) \) for any \( \delta > 0 \). To achieve this improvement, we must run \( \mathcal O ( \|\boldsymbol{a}\|_1/\delta) \) rounds of robust oblivious amplitude amplification. However, we will not further address this detailed aspect, as the focus of this work is on the construction of the quantum circuit itself, rather than other concerns such as quantum state readout.\end{remark}

	\section{Quantum Circuits for gPC Approximation}
	\label{sec:QAprgPC}
	Using \cref{thm:ExpFinDim} and \cref{thm:AlgInfDim}, we derive quantum circuit representations for the \( n \)-term truncated Taylor (\cref{thm:ExpFinDim}) and Chebyshev-polynomial expansions. With the \( n \)-term generalized polynomial chaos (gPC) approximations for the narrow class of \emph{multivariate, holomorphic functions} of Taylor- or Chebyshev-type at hand, 
	we propose constructing PQCs by leveraging the QSP and LCU algorithms.
	
	\subsection{PQC Emulation of Tensorized Taylor polynomials}
	\label{sec:PQCTaygPC}
	For the PQC emulations of $n$-term truncated Taylor-gPC expansions 
	\begin{equation}\label{eq:TaygPC3}
		f^{\operatorname{Tay}}_{\Lambda_n}(\bsy) 
		:= 
		\sum_{\bsnu \in \Lambda_n} t_\bsnu \bsy^\bsnu \;, \quad \bsy \in \cU \;,
		\quad t_\bsnu := \frac{1}{\bsnu!} (\partial^\bsnu_\bsy u)(0) \,
	\end{equation}
	of the (analytic in $[-1,1]^\IN$) 
	in \cref{eq:TaygPC1} we leverage the approach in \cite{yu2023provable}. There, unitaries representing
	(scaled) multivariate monomials $t_\bsnu \bsy^\bsnu$ in \cref{eq:TaygPC3}
	are developed.
	The corresponding PQC's were 
	analyzed in detail in \cite[Lem.~S5]{yu2023provable}.
	We re-state the result for convenience of the reader.
	%
	\begin{lemma}\label{lem:TayQC}
		\cite[Lem.~S5]{yu2023provable}
		Given the monomial $t_\bsnu\bsy^\bsnu$, $\bsy\in [-1,1]^d$, 
		with the scaling $|t_\bsnu\bsy^\bsnu|\leq 1$, 
		and with total degree $|\bsnu|_1 \leq s \in \IN$,
		there exists a PQC $U^\bsnu(\bsy)$ such that 
		\begin{equation}\label{eq:MonPQC}
			\langle + |^{\otimes d} U^\bsnu(\bsy) | + \rangle^{\otimes d} = t_\bsnu \bsy^\bsnu
			\qquad \forall \bsy \in [-1,1]^d\;.
			\;
		\end{equation}
		Furthermore, the PQC $U^\bsnu$ is such that:
		\begin{align*}
			{\rm width}(U^\bsnu) &\leq |\bsnu|_0 \leq d, \\
			{\rm depth}(U^\bsnu) &\leq 2|\bsnu|_0 + 1 \leq 2s+1, \\
			{\rm size}(U^\bsnu) &\leq |\bsnu|_1 + |\bsnu|_0 \leq s+d.
		\end{align*}
		Here, ``${\rm size}$'' of the PQC 
		refers to the number of parameters required to specify the PQC.
	\end{lemma}
	With Lem.~\cref{lem:TayQC} in hand, 
	we emulate the $n$-term truncated Taylor-gPC expansion \cref{eq:TaygPC3} by a PQC constructed 
	via a suitable LCU.
	The following result provides PQCs for total degree $s$ 
	polynomials in dimension $d$ \cite[App.~B, Thm.~1]{yu2023provable}.
	\begin{proposition}\label{prop:TgPCQC}
		For a multivariate Taylor polynomial 
		$p^{\operatorname{Tay}}:[-1,1]^d\to \IR: \bsy \mapsto \sum_{|\bsnu|_1 \leq s} t_\bsnu \bsy^\bsnu$ 
		of total degree $s\in \IN $ such that 
		$\max_{\bsy\in [-1,1]^d} | p(\bsy) | \leq 1$,
		exists a PQC $W_p(\bsy)$ such that 
		\begin{equation}\label{eq:PQCTaygPC}
			p^{\operatorname{Tay}}(\bsy) = \langle 0 | W^\dagger_p(\bsy) Z^{(0)} W_p(\bsy) | 0 \rangle \;,\quad \bsy \in [-1,1]^d \;.
		\end{equation}
		Here, $Z^{(0)}$ is the Pauli observable of the first qubit. 
		The PQC $W_p$ satisfies
		\begin{align*}
			\mathrm{width}(W_p) &= \mathcal O (d + \log(s) + s\log(d)), \\
			\mathrm{depth}(W_p) &= \mathcal O (s^2 d^s (\log(s) + s\log(d))), \\
			\mathrm{size}(W_p) &= \mathcal O (s d^s (s + d)).
		\end{align*}
		Here, the constants hidden in $\mathcal O (\cdot)$ are independent of $d$.
	\end{proposition}
	%
	
	We use Lem.~\cref{lem:TayQC} and Prop.~\cref{prop:TgPCQC} 
	to emulate the Taylor-gPC expansion \cref{eq:TaygPC3} with a PQC. 
	\begin{theorem}\label{thm:TaypQC}
		Let $u:\cU\to \IR$ be $(\bsb,\epsilon)$-holomorphic 
		for some sequence $\bsb \in (0,\infty)^\IN \cap \ell^p(\IN)$ for some $0<p\leq 1$. 
		There exists a constant $C>0$ and a sequence $\{ \Lambda_n \}_{n\geq 1}$
		of nested, d.c. index sets $\Lambda_n\subseteq \cF$ such that for all $n\in \IN$  such that $|\Lambda_n| \leq  n$ ,
		there exists a PQC $W_{n,\Lambda_n}(\bsy)$ 
		such that 
		\begin{equation*}
			| u - \bra{0} W_{n,\Lambda_n}^\dagger Z^{(0)} W_{n,\Lambda_n} \ket{0} | \leq C n^{-(1/p-1)} \;,
		\end{equation*}
		Moreover, the PQC $W_{n,\Lambda_n}$ is such that 
		\begin{align*}
			\mathrm{width}(W_{n,\Lambda_n}) &= \mathcal O (1 + 2 \log(n)), \\
			\mathrm{depth}(W_{n,\Lambda_n}) &= \mathcal O (1 + n \log^2(n)).
		\end{align*}
	\end{theorem}
	\begin{proof}
		Fix $n\in \IN$, and $\Lambda_n \subseteq \cF$ as in \cref{eq:TaygPC3}.
		From Prop.~\cref{prop:TaygPC}, 
		there exists a constant $C>0$ (independent of $n$)
		such that 
		for all $\bsnu\in \Lambda_n$ holds:
		$$
		\max\{|\bsnu |_0, |\bsnu|_\infty \} \leq |\bsnu|_1 \leq C(1+\log(n)) \;.
		$$
		In addition, 
		the sets $\Lambda_n$ are anchored. Hence, 
		\begin{equation}\label{eq:ds}
			d = |\bsnu|_0 \leq C(1+\log(n))\;,\;\;\mbox{and} \;\;
			s=|\bsnu|_1 \leq C(1+\log(n)).
		\end{equation}
		We fix this meaning of $d$ and $s$ through this proof. Choose next $\bsnu\in \Lambda_n$ arbitrary. 
		Lem.~\cref{lem:TayQC} implies that for every $\bsnu\in \Lambda_n$ 
		exists a PQC $U^\bsnu$ such that \cref{eq:MonPQC} holds. From \cref{eq:TaygPC3}, 
		the PQC $W^{\operatorname{Tay}}_n$ emulating the $n$-term truncated Taylor
		expansion $[-1,1]^\IN \ni \bsy \mapsto f^{\operatorname{Tay}}_{\Lambda_n}(\bsy)$ 
		in \cref{eq:TaygPC3} is obtained by summing the PQCs $U^\bsnu$ 
		over all $\bsnu\in \Lambda_n$ via LCU \cite{Childs2012lcu}.
		Similar as in the proof of \cite[App.~B, Thm.~1]{yu2023provable}. Let $n\in \IN$ and $\{ \bsnu(j) \}_{j=1}^n$ be an enumeration of $\Lambda_n$.
		The normalized target operator then reads
		$$
		U_{\Lambda_n}(\bsy) = \sum_{j=1}^n \frac{1}{n} U^{\bsnu(j)}(\bsy) 
		\;.
		$$
		As sums of unitaries $U^{\bsnu(j)}$ are not necessarily unitary,
		to implement $U_{\Lambda_n} (\bsnu)$ via LCUs, one first constructs 
		a unitary $F$ such that 
		$$
		F\ket{0} = \sum_{j=1}^n \ket{j} \;.
		$$
		This $F$ could be e.g. realized by Hadamard gates. With $F$, we build the controlled unitary 
		\begin{equation}\label{eq:CnUntry}
			U_{c,\Lambda_n}(\bsy) 
			:= 
			\sum_{j=1}^n  \ket{j}\bra{j} \otimes U^{\bsnu(j)}(\bsy) \;.
		\end{equation}
		We observe that each $\ket{j}\bra{j} \otimes U^{\bsnu(j)}$ 
		could be constructed
		with $\mathcal O (\log(n))$-qubit controlled Pauli rotation gates, as 
		$U^{\bsnu(j)}(\bsy)$ consists of single-qubit Pauli rotation gates. The $\mathcal O (\log(n))$-qubit controlled gates can be decomposed into
		PQCs of CNOT gates and single-qubit rotation gates resulting in $\mathcal O (\log(n))$ 
		PQC depth, without any use of ancilla qubit, 
		as explained in \cite{daSilva2022LinearDepthMultiQubit}. The resulting unitary 
		$$
		W_{\operatorname{LCU},\Lambda_n} = (F^\dagger \otimes I) U_{c,\Lambda_n} (F\otimes I) 
		$$
		has the property 
		$$
		W_{\operatorname{LCU},\Lambda_n}\ket{0} \ket{+}^{\otimes d}
		=
		\ket{0} U_{\Lambda_n}(\bsy) \ket{+}^{\otimes d} +\ket{\perp}
		\;.
		$$
		Here, 
		$(\bra{0} \otimes I ) \ket{\perp} = 0$. 
		We observe 
		$$
		\bra{0} \bra{+}^{\otimes d} \ket{0}\ket{+}^{\otimes d} 
		=
		\bra{+}^{\otimes d} U_{\Lambda_n}(\bsy) \ket{+}^{\otimes d} 
		=
		f^{\operatorname{Tay}}_{\Lambda_n}(\bsy)\;, 
		\quad 
		\bsy \in [-1,1]^d \;.
		$$
		Hence,
		to numerically evaluate $f^{\operatorname{Tay}}_{\Lambda_n}(\bsy)$, 
		the 
		PQC $\bra{0} \bra{+}^{\otimes d} \ket{0}\ket{+}^{\otimes d}$
		could be estimated with the Hadamard test. We apply the Hadamard test to the $W_{\operatorname{LCU},\Lambda_n}$ which
		results in the PQC $W_{n,\Lambda_n}$ according to Fig.~\cref{figure:lcu-plus-hadamard}.
		\begin{figure}[h]
            \centering
            \sbox0{
            \begin{quantikz}[wire types={b,b},classical gap=0.07cm]
                \gategroup[wires=2,steps=6,style={rounded corners,draw=none,fill=red!40}, background]{$\text{U}_{\operatorname{LCU}} $}
                \lstick{$\ket{0} \; \; $} &  &  & \gate[2]{U_c} &  & \\
                \lstick{$\ket{0} \; \; $} &   & \gate{F} &  & \gate{F^\dagger} &  
            \end{quantikz}
                }%
            \sbox1{
            \begin{quantikz}[wire types={q,b,b},classical gap=0.07cm]
                \gategroup[wires=3,steps=6,style={rounded corners,draw=none,fill=red!40}, background]{$\text{Hadamard Test}$}
                \lstick{$\ket{0} \; \; $} & \gate{H}   &  \ctrl{1}  &  \gate{H}  & \meter{}   & \\
                \lstick{$\ket{0} \; \; $} &  & \gate[2]{U_c}  & &  & \\
                \lstick{$\ket{0} \; \; $} & \gate{F}  &  & \gate{F^\dagger} & &  
            \end{quantikz}
            }%
            \begin{tabular}{cc}
            \usebox0 & \usebox1
            \end{tabular}
            \caption{(Left) The quantum circuit implementing the linear combination of unitaries technique.  (Right) 
            The quantum circuit implementing the Hadamard test.
            }
            \label{figure:lcu-plus-hadamard}
            \end{figure}
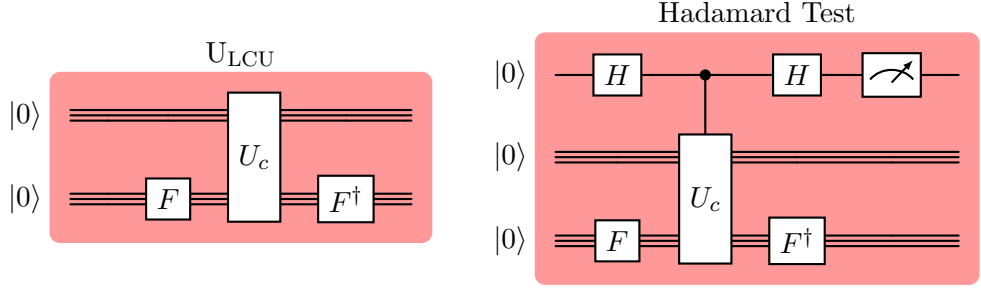

		Measuring the first qubit of $W_{n,\Lambda_n}$ gives
		$$
		f_{W_{n,\Lambda_n}} (\bsy) 
		:=
		\bra{0} W_{n,\Lambda_n}^\dagger(\bsy) Z^{(0)} W_{n,\Lambda_n}(\bsy)\ket{0} 
		=
		\bra{0} \bra{+}^{\otimes d} W_{\operatorname{LCU},\Lambda_n} \ket{0}\ket{+}^{\otimes d} 
		=
		f_{\Lambda_n}(\bsy) 
		\;.
		$$
		The controlled unitary $U_{c,\Lambda_n}(\bsy)$ in the LCU construction  \cref{eq:CnUntry}
		can be implemented by $\mathcal O (ns)$ many $\mathcal O (\log(n))$-qubit controlled gates. An arbitrary $\mathcal O (\log(n))$-qubit controlled gate can be implemented by a PQC 
		made of CNOT gates and single-qubit gates with 
		depth $\mathcal O (\log(n))$ \cite{daSilva2022LinearDepthMultiQubit}. We conclude that $U_{c,\Lambda_n}(\bsy)$ can be implemented by a PQC 
		of depth $\mathcal O (sn\log(n))$ and width $\mathcal O (d+\log(n))$. This gives that depth and width of 
		$W_{\operatorname{LCU},\Lambda_n} = (F^\dagger\otimes I) U_{c,\Lambda_n} (F\otimes I)$
		scale as those of $U_{c,\Lambda_n}$, as $F$ is a tensor product of Hadamard gates. The overall depth of the PQC $W_{n,\Lambda_n}$ is then $\mathcal O (sn\log(n)+d)$, 
		and the overall width of the PQC $W_{n,\Lambda_n}$ is $\mathcal O (d+\log(n))$.
		Inserting \cref{eq:ds}, we arrive at the bounds
		$$
		{\rm depth}(W_{n,\Lambda_n}) \leq C(sn\log(n)+d) \leq C \left((1+\log(n))n\log(n)+1+\log(n)\right) 
		\leq C(1+n\log^2(n))\;,
		$$
		and
		$$
		{\rm width}(W_{n,\Lambda_n}) \leq C(d+\log(n)) \leq C(1+2\log(n)) \;.
		$$
		This implies the depth, width and number of parameters bound.
	\end{proof}

\subsection{PQC Emulation of tensorized Chebyshev-polynomials}
\label{sec:PQCTschgPC}
The core of the ensuing existence and approximation rate results 
is the following multivariate version of \cref{lem:qsp-1d-chebyshev},
which is along the lines of \cite[App.~B]{yu2023provable}.
\begin{lemma}\label{lem:ProdTschQC}
	Let $\bsnu \in \cF $ and with support size $d = |\bsnu|_0 <\infty$,
	and wlog. ${\rm supp}(\bsnu) = [1:d]$. 
	Consider the $d$-variate product Chebyshev polynomial
	\[
	T_{\bsnu}:  [-1,1]^d \to \IR : \bsy \mapsto \prod_{j=1}^d T_{\nu_j}(y_j) \;,
	\]
	with the product denoting pointwise multiplication, 
	at each $\bsy\in [-1,1]^d$. Then, 
	there exists a PQC $W_{\bsnu}$
	such that
	for $\bsy\in [-1,1]^d$
	\begin{equation}\label{eq:QCTnu}
		T_{\bsnu}( \bsy ) = \bra{0} W_{\bsnu}^\dagger(\bsy) Z^{(0)} W_{\bsnu}(\bsy) \ket{0}
		\quad \mbox{for all} \quad \bsy \in [-1,1]^{d}
		\;.
	\end{equation}
	Here, 
	$Z^{(0)}$ is the Pauli $Z$ observable on the first qubit. 
	The PQC is such that: 
	\begin{enumerate}
		\item[(i)] \( \mathrm{Width} (W_{\bsnu}) \) is equal to the number of variables: \( d = |\bsnu|_0 \),
		\item[(ii)] \( \mathrm{Depth} (W_{\bsnu}) \) is given by \( \max \{ 2\nu_1+1, \dots, 2\nu_d+1 \} = 2|\bsnu|_\infty + 1 \),
		\item[(iii)] $\mathrm{Size} (W_{\bsnu})$ is \( |\bsnu|_1 + d \).
	\end{enumerate}
\end{lemma}

\begin{proof}
	\cref{lem:qsp-1d-chebyshev} 
	implies that for each $r\in \IN_0$ exists a PQC $U_{r}$ 
	such that $\bra{0} U_{r}(y) \ket{0} = T_{r}(y)$ for all $y \in [-1,1]$. 
	The width of each $U_{r}(y)$ is $1$, 
	the depth is $2r+1$, 
	and 
	the number of pre-determined parameters is $r+1$. For $\bsnu \in \cF$, 
	consider the unitary
	\begin{equation*}
		U_{\bsnu}(\bsy) 
		= 
		\bigotimes_{j\geq 1} U_{\nu_j} \left( y_j \right).
	\end{equation*}
	Note that, 
	since $U_0 \equiv 1$, with ${\rm supp}(\bsnu) = [1:d]$,
	for all $\bsy \in \cU$ holds
	\begin{equation}\label{eq:TnuUnu}
		\bra{0}^{\otimes d}  U_{\bsnu}(\bsy) \ket{0}^{\otimes d} 
		= 
		\prod_{j=1}^d \bra{0} U_{\nu_j} \left( y_j \right) \ket{0} 
		= 
		\prod_{j=1}^d T_{\nu_j}(y_j) = T_{\bsnu }(\bsy).
	\end{equation}
	The width of $U_{\bsnu }(\bsy)$ is $|\bsnu|_0 \leq d$. Since $U_{\nu_j}$ has depth $2 \nu_j + 1$, 
	the PQC $U_{\bsnu }$ has depth  
	$$
	\max \{ 2\nu_1+1, \cdots, 2\nu_d + 1 \} = 2|\bsnu|_\infty+1.
	$$ 
	Since each $U_{\nu_j}$ has a total of $\nu_j + 1$  pre-determined parameters, 
	the PQC $U_{\bsnu }(\bsy)$ has a total of  
	$\nu_1 + \cdots \nu_d + d = |\bsnu|_1 + d$ 
	pre-determined parameters.
\end{proof}
\subsection{PQC expression of holomorphic maps: finite-dimensional case}
\label{sec:QCe<finty}
We establish PQC expression rates for finite-parametric, holomorphic maps
in the setting of Theorem~\cref{thm:ExpFinDim}, i.e. for parameter domain
$\IU$ of finite dimension $d$.
We write for a multi-qubit system which is initialized in
state $\ket{0}^{\otimes M}$ for some $M\in \IN$ simply $\ket{0}$.
\begin{theorem}\label{thm:QCd<oo}
	Assume $1\leq d < \infty$ and $\IU = [-1,1]^d$.
	Let $f\in \cH(\bsrho)$ for some $\bsrho \in (1,\infty)^d$. 
        Then, there are constants $C,\gamma>0$ (depending on $f$ and $d$) 
	such that, for every $n\in \IN$, 
	exists an LCU PQC $U_n$ such that
	%
	$$
	\| f - \bra{0} U^\dagger_n Z^{(0)} U_n \ket{0} \|_{L^\infty(\IU)} 
	\leq 
	C\exp(-\gamma n^{1/d}).
	$$
	Moreover, 
	the PQC $U_n$ is such that 
	%
	\begin{align*}
		\mathrm{width}(U_n) &= O\left(d + \log_2(n)\right), \\
		\mathrm{depth}(U_n) &= \mathcal O (n^{1 + 1/d} \log_2(n)).
	\end{align*}
\end{theorem}
\begin{proof}
	We first construct the Chebyshev-polynomial approximation. The assumption $d<\infty$ and $f\in \cH(\bsrho)$ 
	allow to use Theorem~\cref{thm:ExpFinDim}. According to that theorem, for every integer $n\geq 2$ exists 
	a d.c. index set $\Lambda_n \subseteq \IN_0^{\IN}$ such that
	the $\Lambda_n$-truncated Chebyshev-expansion 
	\begin{equation}\label{eq:Fln}
		f^{\operatorname{Cheb}}_{\Lambda_n} = \sum_{\bsnu\in \Lambda_n} c_\bsnu T_\bsnu
	\end{equation}
	satisfies the error bound \cref{eq:FinDimErr}. We now construct the PQC. We start the construction of the PQC $U_n$ 
	by emulating $T_\bsnu$ in $f^{\operatorname{Cheb}}_{\Lambda_n}$
	with Lemma~\cref{lem:ProdTschQC},
	resulting in unitary PQCs $\{ U_\bsnu : \bsnu\in \Lambda_n \}$. 
	Next, 
	we build the PQC $U_n$ from the 
	PQCs $U_\bsnu$ with the LCU recapped in \cref{lcu-sec},
	as follows.
	Given that 
	in Theorem~\cref{thm:ExpFinDim} we have
	$\Lambda_n = \{ \bsnu\in \IN_0^d : |\bsnu|_\infty\leq k \}$
	(tensor products of degree $k$ univariate Chebyshev-projections)  with $n = \mathcal O (k^d)$,
	there are $T=n$ terms in \cref{eq:Fln} 
	with respect to some enumeration $j: [1:n]\to \Lambda_n$ of the multi-index set $\Lambda_n$. 
	Next, the LCU assumes at hand an implementation 
	of the controlled unitary
	\begin{equation}
		U_{c,n} := \sum_{\bsnu\in \Lambda_n}  U_\bsnu = \sum_{\bsnu\in \Lambda_n} U_\bsnu \otimes \ket{j}\bra{j} 
		\;.
	\end{equation}
	With the unitary $F$ as in \cref{F-LCU}, 
	with $c_\bsnu$ in place of $a_j$,
	the LCU provides the implementation of the unitary $U_{c,n}$ as follows
	\begin{equation}\label{eq:WLCU}
		W_{\operatorname{LCU},n} := (I \otimes F^{\dagger} ) U_{c,n} (I \otimes F).
	\end{equation}
	With the notation from \cref{lcu-sec},
	for input quantum state \(\ket{\psi}\) 
	and 
	for ancilla qubit \(\ket{0}\),
	\begin{equation}\label{eq:Ucn} 
		(I \otimes F^{\dagger} ) U_{c,n} (I \otimes F) \ket{\psi} \ket{0} 
		= 
		\frac{1}{\| \boldsymbol{c} \|_1} \left( \sum_{\bsnu\in \Lambda_n} c_\bsnu  U_{\bsnu} \right) \ket{\psi} \ket{0} + \ket{\perp}, 
	\end{equation}
	It implies with \cref{eq:TnuUnu} that for all $\bsy\in \IU$
	\begin{equation}\label{eq:fCheb}
		f^{\operatorname{Cheb}}_{\Lambda_n}(\bsy)
		= 
		\sum_{\bsnu\in \Lambda_n} c_\bsnu \bra{0}^{\otimes d}  U_{\bsnu}(\bsy) \ket{0}^{\otimes d}
		=
		\bra{0}^{\otimes d}  
		\left(
		\sum_{\bsnu\in \Lambda_n} c_\bsnu  U_{\bsnu}(\bsy) 
		\right) 
		\ket{0}^{\otimes d}
		\;,
	\end{equation}
	and that the error bound \cref{eq:FinDimErr} holds. It remains to verify the claimed bounds on width and depth of PQC $U_n$.
	The bound on the width of $U_n$ is obvious. We proceed step by step:
	
	\begin{enumerate}
		\item[(i)] First consider the unitary operator
		\begin{equation}\label{equation:state-preperation}
			F\ket{0} = \frac{1}{\sqrt{\| c_{\bsnu} \|_1}}
			\sum_{\bsnu\in \Lambda_n} \sqrt{c_{\bsnu}}\ket{j}.
		\end{equation}
		Here we have assumed that \(c_{\bsnu} > 0\). Otherwise, we can absorbed the negative sign into the corresponding unitary operator \(U_{\bsnu}(\boldsymbol{y})\) without affecting the construction.  Since there are $T \leq n$ terms in the sum 
		in \cref{equation:state-preperation},
		\(F\) can be realized by applying a quantum circuit consisting of \(\mathcal O ( n )\) gates, acting on \(\mathcal O (\log_2 n) \) ancilla qubits.
		
		\item[(ii)] Next consider the controlled unitary operator 
		\[
		U_c(\boldsymbol y) = \sum_{\bsnu \in \Lambda_n} U_{\bsnu}(\boldsymbol y) \otimes \ket{\bsnu}\bra{\bsnu}.
		\]
		This operator \(U_c(\boldsymbol x)\) acts on a composite quantum system that includes \(\mathcal O (d)\) computational qubits and \(\mathcal O (\log_2 n)\) ancilla qubits.
		Each \(U_{\bsnu}(\boldsymbol y)\) can be implemented using a \(\mathcal O (| \bsnu |_\infty )\) single-qubit gates. Since there are \(\mathcal O (n)\) such terms in \(U_c(\boldsymbol y)\), the full operator requires $\mathcal O (n)$ multi-qubit controlled gates controlled on \(\mathcal O (\log_2 n)\) ancilla qubits. 
        The circuit depth required to implement a $\mathcal O (\log_2 n)$  multi-qubit controlled gate is \(\mathcal O (\log_2 n)\). 
        This follows from the argument in~\cite[Sec.~II pp.~3]{daSilva2022LinearDepthMultiQubit}, where it is shown that a multi-qubit controlled gate with control on \( m \) qubits can be implemented with linear depth, i.e., \( \mathcal O (m) \).
		Consequently, the overall implementation of \(U_c(\boldsymbol x)\) 
		involves a PQC with depth scaling as \(\mathcal O (n \log_2(n) | \bsnu |_\infty))\) 
		and width  scaling as \(\mathcal O (d+\log_2(n))\).
		
		\item[(iii)] 
		We utilize the LCU algorithm to construct the operator 
		\begin{equation*}
		W_{\operatorname{LCU}}(\boldsymbol{y}) = W_{\operatorname{LCU}} = (I \otimes F^{\dagger}) U_c(\boldsymbol{y}) (I \otimes F)
		\end{equation*} 
        \item[(iv)] Finally, to compute the desired expectation value, we append a single ancilla qubit to the quantum system and apply the Hadamard test by measuring the first qubit.  This has $\mathcal O (1)$ complexity. This follows directly from the construction in \cref{figure:lcu-plus-hadamard}: we append a single ancilla qubit, apply a constant number of gates (specifically, two Hadamard gates), and perform a single measurement on the ancilla qubit.
		\end{enumerate}
		
		We now estimate the depth of $U_n$.  We start again by observing that Theorem~\cref{thm:ExpFinDim} 
		implies that the tensor-product Chebyshev-approxmations 
		$f^{\operatorname{Cheb}}_{\Lambda_n}$, $n=1,2,3,...$, 
		have coordinate polynomial degree $k=\mathcal O (n^{1/d})$. It holds
		$$
		{\rm depth}(U_n) 
		\leq \max_{\bsnu\in \Lambda_n} \left({\rm depth}(U_\bsnu)\right) 
		\leq \max_{\bsnu\in \Lambda_n} 2 |\bsnu|_1 + 1
		= 1 + 2\max_{\bsnu\in \Lambda_n} |\bsnu|_1
		\;.
		$$
		To bound the latter maximum, 
		we observe that the tensor-product index sets $\Lambda_n\subseteq \IN_0^d$ 
		have, being d.c., 
		$$
		\max_{\bsnu\in \Lambda_n} | \bsnu|_1 
		=
		\max_{\bsnu\in \partial \Lambda_n} | \bsnu |_1 
		= 
		|(k,k,...,k)|_1 = dk = \mathcal O (dn^{1/d})\;.
		$$
		We also have 
		$$
		\max_{\bsnu\in \Lambda_n} | \bsnu|_\infty = k = \mathcal O (n^{1/d}) \;.
		$$
		Here, the constant implied in $\mathcal O (.)$ is independent of $d$.
		This completes the proof.
	\end{proof}
	%
	%
	\subsection{PQC expression of holomorphic maps: infinite-dimensional case}
	\label{sec:QCe=infty}
	We now consider $d=\infty$ case.
	The proof strategy to establish PQC approximation rate bounds  are similar to those used in the proof of Thm.~\cref{thm:QCd<oo}: based on an $n$-term gPC approximation rate bound, 
	here Thm.~\cref{thm:AlgInfDim}, and on Lemma~\cref{lem:ProdTschQC}, 
	we re-approximate the $n$-term truncated Chebyshev-gPC expansion  obtained in Thm.~\cref{thm:AlgInfDim}, in order to deduce  a corresponding PQC with comparable approximation rate bounds. Somewhat modified arguments will be used to investigate the  architecture of the PQCs which arise here: only logarithmic with respect to $n$
	depth circuits will be required, due to \cref{eq:LnlogBd}. Due to parametric-holomorphic maps arising in a wide range of PDE applications with high-dimensional, parametric input, this result
	can be relevant in a wide range of applications. In particular, response surfaces \cref{eq:PDErspsrf} of 
	holomorphic parametric PDE 
	can be efficiently represented by corresponding PQCs.
	\begin{theorem}\label{thm:QcAppdinf}
		Assume given
		$f: [-1,1]^\IN\to \IR$ in $ \cH(\bsb,p,\epsilon)$ 
		with some $0<p\leq 1$ and such that
		$\sup_{\bsy\in [-1,1]^\IN}|f(\bsy)|\leq 1$.  Then, for every $n\in \IN$ there exists a PQC $U_n$ 
		obtained via LCUs of Chebyshev-polynomial circuits  \cref{eq:QCTnu}
		such that the error bounds
		\begin{align*}
			\| f - \bra{0} U^\dagger_n Z^{(0)} U_n \ket{0} \|_{L^2_\varrho(\IU)}
			&\leq C(f) n^{-1/p+1/2}, \\
			\| f - \bra{0} U^\dagger_n Z^{(0)} U_n \ket{0}  \|_{L^\infty(\IU)}
			&\leq C(f) n^{-1/p+1}.
		\end{align*}
		hold.  
		Here $Z^{(0)}$ denotes the first Pauli observable on the first qubit. Moreover, the PQC is such that
		\begin{align*}
			\mathrm{width}(U_n) &= \mathcal O (n), \\
			\mathrm{depth}(U_n) &= \mathcal O (n + \log n).
		\end{align*}
	\end{theorem}
	\begin{proof}
		(Sketch of argument in the case $d=\infty$). The assumptions imply that 
		Thm.~\cref{thm:AlgInfDim} and the error bounds \cref{eq:InfDimErr}
		hold. 
                We construct the approximating PQCs by 
		emulating the truncation $n$-term Chebyshev-gPC expansions in \cref{eq:InfDimErr}. 
                For given $f: [-1,1]^\IN\to \IR$ in $ \cH(\bsb,p,\epsilon)$,
		Thm.~\cref{thm:AlgInfDim} 
		provides a sequence $\{ f_{\Lambda_n} \}_{n\in \IN}$ 
		of $n$-term truncated Chebyshev-gPC expansions 
		\begin{equation}\label{eq:MltvPol}
			f^{\operatorname{Cheb}}_{\Lambda_n}(\bsy) 
			= \sum_{\bsnu\in \Lambda_n} f_\bsnu T_\bsnu(\bsy), 
			\quad \bsy \in \cU = [-1,1]^\IN 
		\end{equation}
		with the approximation error bounds \cref{eq:InfDimErr}.
		Here, 
		the sets $\Lambda_n \subseteq \cF$ are 
		nested d.c. index sets of finite cardinality $|\Lambda_n| \leq n$,
		and $f_\bsnu$ are the Chebyshev-coefficients of $f$. 
		Eqn. \cref{eq:LnlogBd} provides furthermore 
		that there exists a constant $C>0$ such that
		$$
		\forall n\in \IN \; \forall \bsnu \in \Lambda_n : 
		\quad 
		\max_{\bsnu\in \Lambda_n} |\bsnu|_1 \leq C(1+\log(n)) \;.
		$$
		Using 
		$|\bsnu|_0 \leq |\bsnu|_1$ 
		and
		$|\bsnu|_\infty = \max_{j:\nu_j\ne 0} \nu_j \leq |\bsnu|_1$,
		it follows that there exists a constant $C>0$ such that,
		for $q\in \{ 0,1,\infty\}$ it holds
		\begin{equation}\label{eq:NuBd}
			\forall n\in \IN \; \forall \bsnu \in \Lambda_n : 
			\quad 
			\max_{\bsnu\in \Lambda_n} |\bsnu|_q \leq C(1+\log(n)) \;.
		\end{equation}
		We construct the PQCs corresponding to the $\{f_{\Lambda_n}\}_{n\geq 1}$.
		To this end, we proceed as in the proof of Thm.~\cref{thm:QCd<oo}.
		The PQCs representing the $T_\bsnu$ were built in Lemma~\cref{lem:ProdTschQC}.
		The PQC representation of the $f_{\Lambda_n}$ in \cref{eq:InfDimErr} in 
		Thm.~\cref{thm:AlgInfDim} will be via LCU. We observe $\| T_\bsnu \|_{L^\infty(\cU)} = 1 $.
		Assuming $\bsf_n := \{ f_\bsnu \}_{\bsnu\in \Lambda_n}$
		is scaled to ensure $\| \bsf_{n} \|_1 \leq 1$, 
		we note that the approximation \cref{eq:MltvPol} has the form \cref{eq:Ucn}.
		It implies with \cref{eq:TnuUnu}, \cref{eq:fCheb} the PQC emulation 
		of the Chebyshev-gPC expansion \cref{eq:MltvPol}. The estimation of the depth and width of the PQC $U_n := U_{\Lambda_n}$ 
		resulting from LCU construction based on \cref{eq:MltvPol} 
		has, with \cref{eq:NuBd},
		$$
		{\rm width}(U_{\Lambda_n}) 
		\leq 
		\max_{\bsnu\in \Lambda_n} |\bsnu|_0
		\leq 
		Cn
		\;,
		$$
		and
		$$
		{\rm depth}(U_{\Lambda_n})
		\leq 
		\max_{\bsnu\in \Lambda_n} |\bsnu|_0 + 2|\bsnu|_1
		\leq 
		|\Lambda_n| + 2C \log(n) 
		\leq 
		n+2C(1+\log(n))
		\;.
		$$
    This completes the proof.
	\end{proof}
	\section{Conclusions and Extensions}
	\label{sec:Concl}
	For certain classes of multi- and infinite-variate, real-valued maps which admit holomorphic 
	extensions with respect to their arguments into the complex domain, we developed 
	PQC expression rate bounds which are exponential for finitely many parameters, and 
	algebraic, but free from the CoD for functions of countably many variables. Our constructive proofs leverage recent results on $n$-term approximation rates 
	for generalized polynomial chaos (gPC) expansions of such maps, 
	combined with a re-expression of the tensor-product polynomial terms 
	in the gPC expansion by PQCs and LCU.
	Particular emphasis was placed on Taylor- and Chebyshev-gPC expansions, 
	which require bounded parameter ranges, but permit parsimonious 
	product constructions of PQCs.
	
	We considered here only PQC emulations of real-valued, parametric holomorphic maps.
	conceivably, weaker regularity than holomorphy can be sufficient for such results
	(see, e.g. \cite{HSS24}). Also of interest is the PQC emulation of parametric maps with codomains being 
	subsets of separable Hilbert or Banach spaces. Here, PQCs for space discretizations
	have to be developed; this ``quantum circuit expression of PDE solutions'' 
	is, currently, an active area of research, and progress here could be combined 
	with the presently developed results.
	
	All results in the present paper are based on PQC emulations of 
	\emph{finite} tensor products of univariate polynomials.
	In the univariate case, 
	recent significant progress in so-called
	``Infinite QSP'' \cite{alexis2024infinitequantumsignalprocessing,Dong2024infinitequantum,AMT24}
	established a close connection between summability of Chebyshev-coefficient
	sequences and the sequence of (reduced) phase angles in PQCs. 
	Further research could address sparsity of (arrays of) phase angles 
	in the presently considered approximations as a function of the Chebyshev-gPC coefficient
	sequence summability, 
	which determines sparsity and $n$-term approximation rates of finitely truncated gPC expansions.
	
	\section*{Acknowledgement}
	Junaid Aftab acknowledges the support by the National Science Foundation under the grant DMS-2231533.
	Haizhao Yang was partially supported by the US National Science Foundation under awards DMS-2244988, DMS-2206333, the Office of Naval Research Award N00014-23-1-2007, and the DARPA D24AP00325.
	Christoph Schwab acknowledges a visit to the department of mathematics
	at UMDCP in December 2024, during which key arguments in the present paper were obtained.

\printbibliography

\end{document}